\documentclass[12pt]{amsart} 
\usepackage{amsmath, amsthm, amscd, amsfonts,amssymb,graphicx,enumerate}
\numberwithin{equation}{section}
\makeatletter
\@namedef{subjclassname@2010}{%
  \textup{2010} Mathematics Subject Classification}
\makeatother

\newcommand{\teq}{\arabic{section}.\arabic{equation}}

\newcommand{\teql}{\Alph{section}.\arabic{equation}}

\newcommand{\sqr}[2]{{\vcenter{\vbox{\hrule height.#2pt\hbox{\vrule width.#2pt
height#1pt \kern#1pt\vrule width.#2pt}\hrule height.#2pt}}}}

\newcommand{\ssquare}{{\qquad\hfill$\square$}}

\newcounter{eqcount}
\newcounter{ttopic}

\newenvironment{edesc}{\refstepcounter{equation}\begin{enumerate}}%
{\end{enumerate}}

\newcommand{\sdisplay}[1]{{{$\scriptstyle\bullet$}\! #1 \!{$\scriptstyle\bullet$}}} 
\newcommand{\Sdisplay}[1]{{\vskip.2in\noindent $\star$ {\hel #1} $\star$\hskip.1in:}} 

\newcommand{\ring}[1]{{\mathbb #1}}
\newcommand\bZ{{\ring{Z}}}

\newcommand\bC{{\ring{C}}} \newcommand\bR{{\ring{R}}}
\newcommand\bF{{\ring{F}}} \newcommand\bQ{{\ring{Q}}}
\newcommand\bH{{\ring{H}}}
\newcommand{\csp}[1]{{\mathbb #1}}
\newcommand{\tsp}[1]{{\mathcal #1}}

\newcommand{\prP}{\csp{P}}

\newcommand{\sO}{{\tsp{O}}} 
\newcommand{\sQ}{\tsp{Q}}
\newcommand{\sP}{{\tsp {P}}} 
 \newcommand{\sS}{{\tsp {S}}}
 \newcommand{\sH}{{\tsp {H}}}
\newcommand{\sX}{{\tsp {X}}} 
\newcommand{\sM}{{\tsp {M}}} 
\newcommand{\sG}{{\tsp {G}}}

\newcommand{\eql}[2]{{\rm (\ref{#1}\ref{#2})}} 

\newcommand{\vect}[1]{{\pmb #1}} 
 \newcommand{\bg}{\vect{g}}
 
\newcommand{\bp}{{\vect{p}}} \newcommand{\bx}{{\vect{x}}}
 \newcommand{\bw}{{\vect{w}}}
 \newcommand{\bz}{{\vect{z}}}

\newcommand{\row}[2]{{#1_1,\ldots,#1_{#2}}}

\newcommand{\smatrix}[4]{{\big(\begin{array}{cc}
\!\lower2pt\hbox{$\scriptstyle#1$} &\lower2pt\hbox{$\scriptstyle#2$}\!
\\\! \raise2pt\hbox{$\scriptstyle#3$} &\raise2pt\hbox{$\scriptstyle#4$}
\!\end{array}\big)}}
\newcommand{\col}[2]{{\big(\begin{array}{c}
\!\lower2pt\hbox{$\scriptstyle#1$}  \!
\\\! \raise2pt\hbox{$\scriptstyle#2$}
\!\end{array}\big)}}

\newcommand{\texto}[1]{{\textr{#1}}}
\newcommand{\GL}{\texto{GL}} \newcommand{\SL}{\texto{SL}}
 \newcommand{\ind}{\texto{ind}}
\newcommand{\PSL}{\texto{PSL}} 
 
 \renewcommand{\ni}{\texto{Ni}}
 
\newcommand{\textr}[1]{{\text{\rm #1}}}
 \newcommand{\ord}{\textr{ord}}
\newcommand{\abs}{\textr{abs}}  
 
 \newcommand{\inn}{\textr{in}}
 
\newcommand{\pr}{\textr{pr}}

\newcommand{\BCL}{{\text{\rm BCL}}}
\newcommand{\IGP}{{\text{\rm IGP}}}

\newcommand{\rd}{\texto{rd}}

\newcommand{\textb}[1]{{\text{\bf #1}}}
\newcommand{\bfC}{{\textb{C}}}
\newcommand{\longmapright}[2]{\smash{\mathop{\hbox to
#2pt{\rightarrowfill}}\limits^{#1}}}
\newcommand{\longmapleft}[2]{\smash{\mathop{\hbox to
#2pt{\leftarrowfill}}\limits^{#1}}}

\newcommand{\np}{{+}}   \newcommand{\nm}{{-}}

\newcommand{\lrang}[1]{{\langle #1\rangle}}

\newcommand{\eqdef}{\stackrel{\text{\rm def}}{=}}

\newcommand{\pa}[2]{{\frac{\partial #1} {\partial #2}}}

\newfont{\sevenrm}{cmr7}
\newfont{\bsevenrm}{cmbx7}
\newfont{\mathseven}{cmsy7}
\newfont{\bigmath}{cmsy10 scaled 1200}
\newfont{\fiverm}{cmr5}
\newfont{\bfiverm}{cmbx5}
\newfont{\hel}{cmbx10 scaled 1200}
\newfont{\eu}{eufb10}
\newfont{\sseu}{eufm5}
\newfont{\seu}{eufm7}
\newfont{\Cal}{cmmib10}
\newfont{\sCal}{cmmib7}
\newfont{\zch}{eusb10}

\theoremstyle{plain}
\newtheorem{thm}{Theorem}[section] 
\newtheorem{lem}[thm]{Lemma}
\newtheorem{princ}[thm]{Principle}

\newtheorem{prop}[thm]{Proposition}

\theoremstyle{definition}
\newtheorem{defn}[thm]{Definition}
\newtheorem{exmp}[thm]{Example}
\newtheorem{guess}[thm]{Conjecture}

\newtheorem{prob}[thm]{Problem}

\theoremstyle{remark}
\newtheorem{rem}[thm]{Remark}

\newtheorem*{sol}{Solution}

\newcommand{\xs}{\times^s\!}

\def\pic #1 by #2 (#3){\vbox to #2{\hrule width 
#1 height 0pt depth 0pt\vfill\special{picture #3}}}
\def\scaledpicture#1
by #2 (#3 scaled #4){{\dimen0=#1 \dimen1=#2\divide\dimen0 by
 1000 \multiply\dimen0 by #4\divide\dimen1 by 1000 \multiply\dimen1 
by #4\pic \dimen0 by \dimen1 (#3 scaled #4)}}

\newcommand{\comm}[1]{{}}
\newcommand{\bP}{{\tsp {P}}}

\newcommand{\Spin}{{\text{\rm Spin}}}

\renewcommand{\phi}{\varphi}
\newcommand{\Fr}{\text{\rm Fr}}
\newcommand{\MT}{\text{\bf MT}}

\newcommand{\TL}{\text{TimeLine}}

\newcommand{\HM}{\text{\bf HM}}

\newcommand{\OIT}{\text{\bf OIT}}

\renewcommand{\BCL}{\text{\bf BCL}}

\newcommand{\RIGP}{\text{\bf RIGP}}

\newcommand{\rk}{\text{\rm rk}}

\newcommand{\lcm}{\text{lcm}}

\newcommand{\C}{{\text{\rm C}}}
\newcommand{\CM}{\text{CM}}

\newcommand{\ab}{{{}_{\text{\rm ab}}}}
\newcommand{\geng}{{{\text{\bf g}}}}
\newcommand{\sh}{{{\text{\bf sh}}}}

\newcommand{\sF}{{\tsp F}}
\newcommand{\sK}{{\tsp K}}

\newcommand{\Cyc}{{\text{\rm Cyc}}}

\newcommand{\fG}[1]{{\,{}_{#1}\tilde G}} 
\newcommand{\tfG}[2]{{\,{}_{#1}^{#2} G}}

\begin{document}
\baselineskip=17pt
\hoffset.75in

\title[Open Image Theorem]{Introduction to \\ {\sl Monodromy, $\ell$-adic Representations\\ and  The Regular Inverse Galois Problem}}

\author[M.~D.~Fried]{Michael
D.~Fried}
\address{Emeritus, UC Irvine \\ 3547 Prestwick Rd, Billings MT 59101}
\email{mfried@math.uci.edu}

\date{} 
\maketitle

\begin{abstract} There are two famous Abel Theorems. The more well-known, describes the \lq\lq abelian (analytic) functions\rq\rq\ on a one dimensional compact complex torus. The other bundles those complex tori, with their prime degree isogenies. Riemann's generalization of the first features his famous $\Theta$ functions. His deepest work aimed to generalize Abel's second theorem, though he died too young to fulfill that vision. 

To go from Abel to Riemann is a big jump. We go even further in $$\begin{array}{c}\text{extending isogonies of complex torii $\Longrightarrow$} \\ \text{towers of modular curves $\Longrightarrow$ {\bf M}(odular){\bf T}(ower)s.}\end{array}$$

\S\ref{WhatGauss}  starts with 1-variable rational functions (as in junior high). With these we introduce {\sl Nielsen classes\/} attached to $(G,\bfC)$  ($\bfC$ is $r$ conjugacy classes in a finite group $G$) and a {\sl braid action\/} on them. These give {\sl reduced Hurwitz\/} spaces, denoted $\sH(G,\bfC)^{\rd}$.  

For any prime $\ell$ dividing $|G|$ that has no $\bZ/\ell$ quotient, there are many choices of $\bfC$, giving a {\sl canonical\/} tower of reduced Hurwitz spaces over $\sH(G,\bfC)^{\rd}$, using the {\sl Universal Frattini cover\/}, $\tilde G$, of $G$ and  $\tilde G_\ab$, its {\sl abelianized\/} version. The towers are nonempty assuming $\bfC$ are $\ell'$ classes satisfying a cohomological condition from a {\sl lift invariant}. 

{\sl The\/} first case --  $G=D_\ell$, the order $2\ell$ ($\ell$ an odd prime) dihedral group, $\bfC$ four repetitions of the involution class -- gives modular curve towers. 

General tower levels, $\sH_k$, $k\ge 0$, naturally referenced by  $\ell^{k\np1}$, share the {\sl moduli\/} properties -- tower level points represent algebraic objects -- of modular curves. For {\sl all\/} 1-dimensional examples ($r=4$):  
$$\text{$\sH_k$ is an upper half-plane quotient ramified over the $j$-line at $0,1,\infty$.}$$    

We introduce the book \cite{Fr18}. It shows how \MT s expands applications for modular curves by interpreting problems directly from the $$\text{{\bf R}(egular) {\bf I}(nverse) {\bf  G}(alois) {\bf  P}(roblem).}$$ The book concludes by showing that  Serre's \OIT\ theorem -- dividing decomposition groups on a modular curve tower into two disjoint types, $\CM$ and $\GL_2$ -- has a formulation on any \MT. 

Characteristic $\tilde G$ quotients and attached modules  tells us about cusps, definition fields and  rational points on the tower levels over  $\sH(G,\bfC)^\rd$. From $\tilde G_\ab$ come the title's $\ell$-adic representations, guided by the \RIGP. 
\end{abstract} 

\setcounter{tocdepth}{3} 
\tableofcontents

\section{Changing from isogonies to sphere covers} \label{WhatGauss}  \S\ref{whyrationalfuncts} uses rational function covers of the sphere, $\prP^1_z$, in an inhomogeneous complex variable $z$. Then,   despite the use of such elementary objects, this generalizes to any covers of the sphere to produce all we need. The starting objects, {\sl Nielsen classes\/} attached to $(G,\bfC)$, of \cite{Fr18}:  $G$ a finite group, and $r$ conjugacy classes, $\bfC$, in $G$.  Thereby we introduce the basic moduli -- reduced {\sl Hurwitz\/} -- spaces, $\sH(G,\bfC)^\rd$ that are level 0 of a  ({\bf M}(odular) {\bf T}(ower)) of moduli spaces,  the subject of \cite{Fr18}. 

Before, however, going through our elementary approach, \S\ref{bual} goes explains why changing from isogenies of torii, to the braid action describing families of sphere covers, opens new territory to old problems. 

\subsection{Backing up a little} \label{bual} \MT s generalize modular curve towers. As is well-known, modular curves are moduli spaces for elliptic curves and their torsion, defined by congruence subgroups of $\SL_2(\bZ)$.  Less well known, they are the image of moduli for {\sl certain\/} covers of $\prP^1$. Further, the induced map on the moduli spaces is one-one, losing none of the abelian variety data.  

Seeing elliptic curves as moduli of dihedral group covers is in the spirit of \MT s. It connects problems studied by many directly to $$\text{Serre's {\bf O}(pen){\bf I}(mage){\bf T}(heorem). }$$ We remind of the history of two such abelian variety problems with totally independent application  in \S\ref{pre95} and \S\ref{95-04}. I illustrate  the format for our discussions, as applied to the \OIT\ using the symbol \sdisplay{\cite{Se68}}, indicating Serre's book. The year telegraphs that it is in \S\ref{pre95}. There a reader will find a setoff $\star$-display, here $\star$ {\cite{Se68}} $\star$, starting an elaboration  on how Serre's book relates to our topics.    

\MT s generalizes this particular connection in the same sense that general finite groups generalize dihedral groups. The territory opens up because of the {\sl automatic braid action\/} on Nielsen classes \S\ref{equivalences}. \S\ref{oitpre}, however, makes an elementary point, by \lq\lq backing up\rq\rq\ to the moduli of sphere covers. 

The rich territory of \MT s goes beyond the abelian monodromy action that dominates from congruence subgroups, though that remains a computational tool.  The towers that generalize modular curve towers work through the {\sl Universal $\ell$-Frattini cover\/} of a finite group. \S\ref{univfratpre} is a prelude to the Frattini central extensions that reveal new motivic aspects, and the canonical $\ell$-adic representations that appear in the title. 

\subsubsection{Prelude to generalizing the \OIT} \label{oitpre} In expanding to Serre's $\OIT$, rather than directly going from the space of 1-dimensional abelian varieties to the space of higher dimensional abelian varieties, the \MT\ method goes through spaces of curve covers. Yet, it still connects canonically  to data on their Jacobians through Frattini covers. The indirect route  expands to problems both more elementary and more advanced. 

Further, the direct method runs into problems from correspondences from curves that are invisible to the process, without a method to handle them. The indirect approach has two advantages.
 
 \begin{edesc} \label{wwd} \item \label{wwda} It gives tools from finite group theory and homological algebra to see the motivic pieces that contribute to expansions of the OIT. 
\item \label{wwdb} In using sphere covers as the moduli problem it connects to a wider range of classical problems than does the direct method. \end{edesc} 

In \eql{wwd}{wwda}, the motivic pieces are essentially synonymous with braid orbits on the Nielsen class. That is, without some examples,  even very enlightened mathematicians might not detect a substantial deformation difference between two covers of $\prP^1_z$ that have precisely the same group, permutation representation and conjugacy classes. 

Yet, such substantial deformation differences -- it is impossible to deform one-to-the-other by deforming their branch points (while keeping those separate) -- do occur, often significantly relevant to problems that go back a long way. The discussion on \sdisplay{\cite{Se90a}} and \sdisplay{\cite{Fr10}} mirrors an astonishingly long history,  back to Riemann. 

This doesn't, at first, seem to be what {\sl motivic\/} has often meant, subspaces of some $\ell$-adic cohomology separated by monodromy action. Yet, it is! Braid action just catches it quicker, at level 0 of a \MT. 

Often, as in the case we are citing it is caught homologically, by  the covers being attached to distinct elements of $H_2(G,\bZ)$. Even if you don't care about any \MT\ level except 0, that data is there, automatically. Distinct braid orbits are not yet easy to detect. Therefore much has gone into detecting them computationally, under the rubric called {\sl lift invariants\/} for which our discussion includes two examples: above, and the lift invariants attached to the main example extending the \OIT\ in \cite{Fr18}. In this latter case, the distinction in motivic pieces shows transparently. 

For \eql{wwd}{wwdb}, I refer to discussions in  \S\ref{earlyOIT-MT} and \S\ref{Frat-Groth} that connect {\sl involution realizations\/} of dihedral group on one hand, to problems with a large literature that one can explain to undergraduates at most US research universities. They give an alternative path to the great mysteries of algebra/analytic geometry, even as they show how one might forge progress. 

The first of these connects involution realizations with cyclotomic points on hyperelliptic jacobians. Further, \MT s completely and (relatively easily) generalizes that problem starting from essentially any finite group G and prime $\ell$ for which it is $\ell$-perfect (see \eqref{MTconds}).  This is as in \sdisplay{\cite{DFr94}}, also discussed in the \cite{Fr94} (reprinted in several languages by Serre). 

Further, that generalization starts with a statement about the \RIGP\ that is Modular Tower free. Yet, it  forces the existence of Modular Towers with a diophantine problem (no rational points at high levels), that generalizes no points on high modular curve levels. 
 
We can compute properties of the  tower levels from assiduous use of an Artin braid group quotient: the {\sl Hurwitz  monodromy group\/} $H_r$ (and its signficant subgroups). \S\ref{braids} defines \MT s. 

When $r=4$,  Thm.~\ref{genuscomp} computes a first property, the genus, of these 1-dimensional spaces that appear as $j$-line covers. Conspicuously, it requires  no details on a subgroup of $\PSL_2(\bZ)$ -- a congruence subgroup only in extremely special cases -- which defines the space as a quotient of the complex upper-half plane. This formula in one tool by which we prove properties of the 1-dimensional \MT s. 

\subsubsection{Universal $\ell$-Frattini covers} \label{univfratpre} 
An abstract for \cite{Fr18} (\S\ref{Fr18abstract}),  expands to explain the title, \lq\lq What Gauss told Riemann \dots ,\rq\rq\  of \S\ref{S1title}. That concludes in \S\ref{algvsanal}. Finally, \S\ref{Frat-Groth} introduces the analogy by which a Hurwitz space $\sH(G,\bfC)^\rd$ generalizes the moduli, $\sM_{\geng,r}$, of $r$ unordered marked points on curves of genus $\geng$.  With so many applications when $\geng =0$,  we stay with that case. The \MT\ idea works just as well in the general case. 

\begin{defn} \label{frattdef} A cover of groups, $\psi: H\to G$, is  {\sl Frattini\/} if, given a subgroup $H^*\le H$, should $\psi(H^*)=G$, then $H^*=H$. 
\end{defn} There is a Universal profinite group, $\tilde G$,  for the Frattini covering property. Further, for each prime $\ell$ dividing $|G|$, there is a universal ${}_\ell\tilde G$ for the Frattini property with kernel an $\ell$ group. A characteristic sequence  $\{\tfG \ell k\}_{k=0}^\infty$ of quotients of ${}_\ell \tilde G$ canonically defines a series of moduli space covers of $\sH(G,\bfC)^\rd$ when the elements of $\bfC$ are prime to $\ell$. This is elementary and explained in  \sdisplay{\cite{Fr95}} and \cite[App.~B]{Fr18}. 

Still, except when the $\ell$-Sylow of $G$ is normal, it is no easy guess as to what is $\fG \ell$, though its rank (minimal number of generators) is the same as that of $G$.  
\S\ref{explainingFrattini} describes several papers that were preludes to \cite{Fr18}. 

As in \S\ref{dragging}, denote the sphere, $\prP^1_z$, punctured at $\bz=\{\row z r\}$, by $U_\bz$, and its fundamental group (modulo inner conjugations) by $\pi_1(U_\bz)$.  Suppose $\ell$ is a prime for which $G$ has no $\bZ/\ell$ quotient: $G$ is {\sl $\ell$-perfect}.  Then,  \S\ref{modtowdef} defines a set of $\pi_1(U_\bz)$ quotients, $\sF_{G,\bfC,\ell}$, from $(G,\bfC)$ with image isomorphic to $\fG \ell$ (or of  $\tfG \ell {} {}_\ab$). It also defines an $H_r$ action on them.  

A cohomological condition must be satisfied to assure $\sF_{G,\bfC,\ell}$ is nonempty. A \MT\ on $\sH(G,\bfC)$ is an $H_r$ orbit in this action. It is also a projective system of Hurwitz space components. 

 \cite{Fr18} could not hold all that material, even if appropriate. Instead, \sdisplay{\cite{Fr18}}  relates Frattini covers and the Inverse Galois Problem. This paper describes Serre's original \OIT\ in \S\ref{pre95} in the discussion of \sdisplay{\cite{Se68}} and \sdisplay{\cite{Fr78}}, updated from the original papers with references to \cite{Fr05}, \cite{GMS03} and \cite{Se97b}. Especially, this gives background  on the problems that connected our Hurwitz space approach  to the \OIT. 

\cite{Fr18} is quite complete on the Universal Frattini cover itself, and especially the role of  the {\sl lift invariant}.  Nontrivial lift invariants arise from what group theorists call representation covers of $G$. They are also a detectible subset of central (in the notation of Def.~\ref{frattdef}, when $\ker(\psi)$  is in the center of $H$) Frattini covers. This set of ideas gives the main information we require to understand the \MT\ levels,  components, their cusps and why they are appropriate for generalizing the OIT.   

\subsection{Part I: Data for covers of the sphere} \label{whyrationalfuncts} 

Here is notation for a rational function: $f: \prP^1_w \to \prP^1_z: w \mapsto f(w)$:  $$f = f_1(w)/f_2(w),\text{ with $(f_1,f_2)=1$, $n=\deg(f)=\max(\deg(f_1),\deg(f_2)$. }$$
$$\text{For example: $\frac{w^3\np1}{w^5\nm3}$ has {\sl degree}, $\deg(f)$, 5.}$$  

{\sl Branch points}: Places $z'$ where distinct $w' \mapsto z'$, $\row {w'} t_{z'}$, have cardinality $t_{z'}<n$. Take $z'=0$. For simplicity assume $$\text{no $w'= \infty$ ($\deg(f_1)\ge \deg(f_2)$), and  the $f_i\,$s have leading coefficient 1.}$$ We write solutions, $w$, of $f(w)=z$, as analytic functions in a recognizable variable.  For $1\le k\le t$, write $f(w) = (w - w_k')^{e_k} m_k(w)$ with $m_k(w_k') \not = 0$. Form an expression in the variable $z^{1/e_k} = u_k(z)$: 
$$w_k(z^{1/e_k})\eqdef w_k' + u_k(z) + a_2 u_k(z)^2  + a_3 u_k(z)^3  \dots = w_k(u_k(z)), k = 1,\dots, t_{z'}. $$ 

Substitute $w_k(z^{1/e_k})\mapsto w$ in $f(w)=z$. Look at leading powers of $u_k(z)$ on the left and on the right. They are equal. Now solve inductively for $a_2, a_3, \dots$, so  the left side is identically equal to $z$. Easily verify these.  
\begin{edesc} 
\item The result for $w_k(z^{1/e_k})$ is analytic in a neighborhood of $u_k(z) = 0$. 
\item With $\zeta_{m}=e^{2\pi \sqrt{-1}/m}$, substitution(s) $u_k(z) \mapsto \zeta_{e_k}^j u_k(z)$, \\ $j= 1,\dots, e_k$,  give $e_k$ distinct solutions $w\in \bC((u_k(z)))/\bC((z))$.  \end{edesc} 
Take $\bar e\eqdef \bar e_{z'} = \lcm(\row e t_{z'})$. Now, put those solutions together. 

Write all $w_k(\zeta_{e_k}^jz^{1/e_k})\,$s, $k=1,\dots,t_{z'}$, as power series in $z^{1/\bar e}$:  
$$\text{ Substitute the obvious power of $z^{1/\bar e}$ for each  $u_k(z)$.}$$ 
This gives $n$ distinct solutions, $L_{z'}$, of $f(w) = z$ in the field $\bC((z^{1/\bar e}))$ and a natural permutation on $L_{z'}$ from the substitution  
\begin{equation} \label{powseries} \hat g_{z'}: z^{1/\bar e} \mapsto e^{2\pi i/\bar e} z^{1/\bar e}. \end{equation} 
This gives an element in  $S_n$ (the identify, if all $e_k\,$s are 1), the symmetric group, on the letters $L_{z'}$.\footnote{We are abusing notation by still having $z'=0$, but we are about to drop that by using $\row z r$ for a labeling of the branch points.}

Do this for each branch point, $\row {z'} r$, to get $\bg\eqdef (\row g r)$. 
\begin{edesc} \label{branchcyc} 
\item  \label{branchcyca}  How can we compare entries of $\bg$,  by having them all act on one set of letters, rather than on $r$ different sets, $L_{z_1},\dots, L_{z_r}$? 

\item  \label{branchcycb} With success on \eql{branchcyc}{branchcyca}, denote the group $\lrang{\bg}$ generated by $\bg$  by $G_f$. What can we use it and $\bg$ for? 
\item \label{branchcycc} Was there anything significant about using rational functions?  
\end{edesc} 

In this section I answer questions \eqref{branchcyc}. In \S\ref{deform}, with those answers, to assure $\bg$ are appropriate,  as a preliminary, we extend to do the following. 
\begin{equation} \label{deform} \text{Invert this, to get $\bg \mapsto f=f_\bg$, regardless of the branch points.}\end{equation}  

Now I produce $G_f$, up to isomorphism as a subgroup of $S_n$, answering \eql{branchcyc}{branchcyca}, first using {\sl Algebraic Geometry\/}, then using {\sl Analytic Geometry}. 

\subsubsection{Algebraic Geometry} \label{alggeom} We  introduce a compact Riemann surface cover, $\hat f:  \hat W \to \prP^1_z$, the {\sl Galois closure\/} of $f$, minimal with these properties:  
\begin{edesc} \label{galoisClos} \item $\hat f$ factors through $f$; and   
\item \label{galoisClosb} it is a {\sl Galois cover\/} of $\prP^1_z$. 
\end{edesc} The group of those automorphisms is $G_f$.  
Galois cover in \eql{galoisClos}{galoisClosb} means having $\deg(\hat f)$ automorphisms commutating with $\hat f$. 

Form the fiber product of $f$, $\deg(f)=n$ times (assume $n>1$): $$(\prP^1_w)_f^{(n)} \eqdef \{(w_1,\dots, w_n) \in (\prP^1_w)^n \mid f(w_1)=\cdots = f(w_n) \}\text{ over $\prP^1_z$.}$$ The resulting object is {\sl singular\/} if two coordinates, $w_k',w_l'$, lying over the same branch point $z_i,$ have both $e_k>1$ and $e_l > 1$. To see this, project onto the $(k,l)$ coordinates through the point $(w_k',w_l')\in \prP^1_w\times_{\prP^1_z}\prP^1_w$. Around $(w_k',w_l')$, this space (with its map $f$) is locally analytically isomorphic to $$\{ (w_k,w_l) \in D_{w'_k=0} \times_{D_{z'=0}} D_{w'_l=0} | w_k^{e_k} - w_l^{e_l} = 0 \}\to D_{z'=0}.$$  

The {\sl Jacobian criterion\/} reveals the singularity at $(0,0)$:  both partials $\pa{\ } {w_k},\pa{\ } {w_l}$ of $w_k^{e_k} - w_l^{e_l}$ are 0 at (0,0). Indeed,  normalizing this as a cover of $D_{z=0}$ results in $\gcd(e_k,e_l)$ copies of the cover $w \mapsto w^{\lcm(e_k,e_l)}$.\footnote{A disk, $D_{z'=0}$, around $z'=0$, is a convenient open set for us, as we see in \S\ref{braids}. Technically, that means someone has selected a metric on, say, $\prP^1_z$.} 

Denote the normalization by ${}^*W^{\{n\}}_f$. The normalization may have several components. One for certain is the {\sl fat diagonal}, $$\Delta_f^{\{n\}} =  \text{closure of the locus where 2 or more of those $w_i \,$s are equal.}$$ Remove the components of $\Delta_f^{\{n\}}$. On the result, there is a natural action of $S_n$, by permuting those distinct $w_i\,$ s, that extends to the whole normalized (since 1-dimensional, nonsingular) ${}^*W^{\{n\}}_f$. 

If ${}^*W^{\{n\}}_f$ is irreducible, then it is a Galois over $\prP^1_z$ with group $S_n$. If it is {\sl not\/} irreducible, consider a component, $\tilde f: \tilde W_f\to \prP^1_z$. Then, $$G_f = \{ g\in S_n | g \text{ preserves }  \tilde W_f\}, \text{ is the {\sl geometric\/} monodromy group.}$$  Automatically $|G_f |= \deg(\tilde  f)$. 

Denote the conjugacy class of $g\in G_f$ by $\C_g$. Though $\bg=(\row g r)$ still depends on how we labeled points over branch points, this approach does define the conjugacy classes $\C_{g_i}$, $i= 1,\dots, r$, in $G_f$: consider maps of the function field of $W_f$ into $\bC(((z-z_i)^{1/{\bar e_i}}))$ fixed on $\bC((z-z_i))$ and restrict the automorphism $\hat g_{z_i}$ of \eqref{powseries} to the image of that embedding. 

For many purposes, this construction is inadequate to that of \S\ref{analgeom}. Since, however, it is algebraic, it gives another group we need.  

\begin{defn} \label{arithmon}  If  the cover $f$ is defined over a field $K$, then $\hat G_{f,K}\eqdef \hat G_f$,   the {\sl arithmetic monodromy\/} of $f$ (over $K$), is defined exactly as above, except take $\hat f: \hat W \to \prP^1_z$ to be a component defined over $K$. \end{defn} That is, if we started with $\tilde W$, a geometric component, then we would take for $\hat f: \hat W \to \prP^1_z$, the union of the conjugates $\tilde f^\sigma: \tilde W^\sigma\to \prP^1_z$, $\sigma \in G_K$. This would be defined and irreducible over $K$. 

\subsubsection{Analytic Geometry}  \label{analgeom} For $g\in S_n$ with $t$ disjoint cycles, its {\sl index\/} is  $\ind(g)=n \nm t$. 
For $\bfC$, $r$ conjugacy classes -- some may be repeated, count them with multiplicity -- in a group $G$, use $\bg\in G^r\cap \bfC$ to mean  an $r$-tuple $\bg$ has entries in {\sl some\/} order (with correct multiplicity) in $\bfC$. We denote the group the entries generate by $\lrang{\bg}$. 

\begin{exmp} If $G=S_4$, and $\bfC=\C_{2^2}\C_{3^2}$ consists of two repetitions each of the class, $\C_2$, of 2-cycles, and the class, $\C_3$, of 3-cycles, then  both $${}_1\bg=((1\,2), (2\,3\,4),(3\,4), (1\,3\,4))\text{ and } {}_2\bg=((2\,3\,4),(1\,3),(1\,3),(3\,2\,4)) $$ are in  $(S_4)^4\cap \bfC$. Both satisfy \eql{bcycs}{bcycsa}, but only ${}_2\bg$ satisfies \eql{bcycs}{bcycsb}. \hfill $\triangle$
\end{exmp} 

Our next approach to the Galois closure, based on Thm.~\ref{BCYCs}, gives us a better chance to answer questions \eql{branchcyc}{branchcyca}-\eql{branchcyc}{branchcycc}.  To simplify notation, unless otherwise said, always make these two assumptions. 
\begin{edesc} \label{assumptions} \item  \label{assumptionsa} Conjugacy classes, $\bfC=\{\row \C r\}$, in $G$ are {\sl generating}. 
\item \label{assumptionsb} $G$ is given as a transitive subgroup of $S_n$.  \end{edesc} 
 Meanings: \eql{assumptions}{assumptionsa} $\implies$ the full collection of elements in $\bfC$ generates $G$; and \eql{assumptions}{assumptionsb} $\implies$ the cover generated by $\bg$ is connected. Even with \eql{assumptions}{assumptionsa}, it may be nontrivial to decide if there is $\bg \in G^r\cap\bfC$ that generates.  

\begin{thm} \label{BCYCs} Assume  $\bz\eqdef \row z r\in \prP^1_z$ distinct. Then, some degree $n$ cover $f: W \to \prP^1_z$ with branch points $\bz$, and $G=G_f\le S_n$ produces  classes  $\bfC$  in $G$, if and only if there is $\bg\in G^r\cap \bfC$ with these properties:  
\begin{edesc} \label{bcycs}   \item  \label{bcycsa}   $\lrang{\bg} = G$ ({\sl generation}); and 
\item  \label{bcycsb}   $\prod_{i=1}^r g_i = 1$ ({\sl product-one}). 
\end{edesc} 
Indeed, $r$-tuples satisfying \eqref{bcycs} give all possible Riemann surface covers -- both up to equivalence  (see \S\ref{equivalences})-- with these properties. 

Refer to one of those covers attached to $\bg$ as $f_\bg: W_\bg \to \prP^1_z$. 
\begin{equation} \label{RH} \text{The genus $\geng_{\bg}$ of $W_\bg$ appears in } 2(\deg(f) \np \geng_\bg -1) = \sum_{i=1}^r \ind(g_i). \end{equation}  \end{thm} 

The set of $\bg$ satisfying \eqref{bcycs} are the Nielsen classes associated to $(G,\bfC)$. A Riemann surface $W_\bg$  is isomorphic to $\prP^1_w$ over $\bC$ if and only if formula \eqref{RH} -- {\sl Riemann-Hurwitz\/} --  gives $\geng_\bg=0$. 

Denote $\prP^1_z\setminus \{\bz\}$ as $U_{\bz}$ and choose $z_0\in U_{\bz}$. Thm.~\ref{BCYCs} follows from existence of {\sl classical generators\/} of $\pi_1(U_{\bz},z_0)$. These are paths $\sP=\{\row P r\}$ on $U_\bz$ based at $z_0$, of form $\lambda_i\circ\rho_i\circ \lambda^{-1}_i$ with these properties. 
\begin{edesc} \label{generatorspi} \item  $\rho_i\,$s are non-intersecting clockwise loops around the respective $z_i\,$s. 
\item The $\lambda_i\,$s go from $z_0$ to a point on $\rho_i$. 
\item  Otherwise there are no other intersections.  
\item The $\row \lambda r$ emanate clockwise from $z_0$. 
\end{edesc}

Suppose a cover, $f:W\to\prP^1_z$, has a labeling of the fiber $w_{1}^\star,\dots w_{n}^\star$ over $z_0$. Then, analytic continuation of a lift, $P^\star_{k,i}$, of  $P_i$, starting at $w_{k}^\star$ will end at a point which we call $w_{(k)g_i}^\star$, on $k\in \{1,\dots,n\}$. 

This produces the permutations $\row g r$ satisfying \eqref{bcycs}.  This results from knowing \eqref{generatorspi} implies $\row P r$  are generators of $\pi(U_\bz,z_0)$, and  
\begin{equation} \label{fundgp} \text{they have product 1 and {\sl no other relations\/}.} \end{equation} 
From \eqref{fundgp}, mapping $P_i\mapsto g_i$, $i=1,\dots,r$, produces a permutation representation $\pi(U_\bz,z_0)\to G\le S_n$. 

From the theory of the fundamental group, this gives a degree $n$ cover $f^0: W^0\to U_\bz$. 
Completing the converse to Thm.~\ref{BCYCs} is not immediate. You must fill in the holes in $f_0$ to get the desired $f: W\to \prP^1_z$. A full proof, starting from from \cite{Ahl79}, is documented in \cite[Chap.~4]{Fr80}. 

\begin{defn} \label{NielsenClass} Given $(G,\bfC)$, the set of $\bg$ satisfying \eqref{bcycs} is the {\sl Nielsen class\/}  $\ni(G,\bfC)^\dagger$, with $\dagger$ indicating an equivalence relation referencing the permutation representation $T: G \to S_n$. \end{defn}
A cover doesn't include an ordering its branch points. Adding such an ordering would destroy most applications number theorists care about. This makes sense of saying {\sl a cover is in the {\sl Nielsen class\/} $\ni(G,\bfC)^\dagger$}. 

\begin{rem}[Permutation notation] \label{permnot} Our usual assumptions start with a faithful transitive permutation representation $T: G\to S_n$ with generating conjugacy classes $\bfC$, from which we may define a Nielsen class $\ni(G,\bfC)^\dagger$. 

Or, if $T$ comes from a cover $f:W\to \prP^1_z$, denote the permutation presentation by $T_f$ even when applied to $\hat G_f$. Then, denote the group of $\hat W/W$ by $\hat G_f(1)$. That indicates it is the subgroup stabilizing the integer 1 in the representation, and $G_f(1)=\hat G_f(1)\cap G_f$. \end{rem} 

\subsection{Part II: Braids and deforming covers} \label{braids}  

Consider $U_r$, subsets of $r$ distinct unordered elements $\{\bz\}\subset \prP^1_z$: 
\begin{equation}\label{discriminant}  U_r=\text{projective $r$-space $\prP^r$ minus its {\sl discriminant locus\/} ($D_r$).} \end{equation}  
Take ${}_0\bz$ to be a basepoint of $U_r$, and denote $\pi_1(U_r,{}_0\bz)$, the Hurwitz monodromy group, by $H_r$. A {\sl Hurwitz space\/} is a cover of $U_r$ that parametrizes all covers in a Nielsen class. \S\ref{dragging} explains how it comes from a representation of $H_r$ on a Nielsen class. 

\subsubsection{Dragging a cover by its branch points} \label{dragging}  Here is the way to think of forming a Hurwitz space. 
Start with  ${}_0f: {}_0W\to \prP^1_z$, with branch points ${}_0\bz$, classical generators ${}_0\bP$ and (branch cycles) ${}_0\bg\in \ni(G,\bfC)$.  

Drag ${}_0\bz$ and ${}_0\bP$, respectively,  to ${}_1\bz$  and ${}_1\bP$ along any path $B$ in $U_r$. 
With no further choices, ${}_t\sP\mapsto {}_0\bg$ forms a trail of covers  
${}_tf: {}_tW\to \prP^1_z$, $t\in [0,1]$, $$\text{{\sl with respect to the same ${}_0\bg$ along the path indicated by the parameter.}}$$ 

This produces a collection of $\prP^1_z$ covers of cardinality $|\ni(G,\bfC)^\dagger|$ over every point of $U_r$, forcing upon us a decision.  
For $B$ closed, denote the homotopy class $[B]$ as $q_B\in H_r$. 

\begin{princ} For $B$ be a closed path, identify branch cycles ${}_1\bg$ for the cover ${}_1f: {}_1W\to \prP^1_z$ lying at the end of the path, relative to the original classical generators ${}_0\bP$ from ${}_0\sP \mapsto ({}_0\bg)q_B^{-1}$. \end{princ}

Here are key points going back to \cite[\S4]{Fr77}. 
 \begin{edesc} \label{keys} 
\item \label{keysa}  {\sl Endpoint of the Drag}: A cover at the end of $B$ is still in $\ni(G,\bfC)^\dagger$; it depends only on the homotopy class of $B$ with its ends fixed.  
\item \label{keysb}  {\sl $H_r$ orbits}: (Irreducible) components of spaces of covers in $\ni(G,\bfC)^\dagger$ correspond to  braid ($H_r$) orbits. 
\end{edesc}  

Whatever the problem application, we must be able to identify the Galois closure of the cover. The key  ambiguity is in labeling $\bw^\star=\{w_1^\star,\dots,w_n^\star\}$, points lying over $z_0$. Changing that labeling changes $T: G\to S_n$. A slightly subtler comes from changing $z_0$. There is a distinction between them. Changing $z_0$ to $z^*_0$ is affected by rewriting the $z_i$-loops as \begin{equation} \label{baseptch} \lambda^*\circ\lambda\circ\rho\circ \lambda^{-1}\circ(\lambda^*)^{-1},  \text{ with $\lambda^*$ a path from }z_0^* \text{ to }z_0.\end{equation}

\subsubsection{Braids and equivalences} \label{equivalences} It is immediately helpful having a natural set of generators of $H_r$ giving their action on $\ni(G,\bfC)$. 

\begin{edesc} \label{Hrgens} \item  \label{Hrgensa} $H_r$ is generated by two elements: 
 $$\begin{array}{rl} q_i: & \bg\eqdef (\row g r)  \mapsto (\row   g 
{i-1},g_ig_{i+1}g_i^{-1},g_i,g_{i+2},\dots, g_r); \\ 
\sh:  &\bg\mapsto (g_2,g_3,\dots,g_r,g_1) \text{ and } H_r  \eqdef \lrang{q_2,\sh} \text{ with }\\ &\sh\,q_i\ \sh^{-1} = q_{i\np1},\   i=1,\dots,r\nm 1. \end{array}$$ 
\item \label{Hrgensb} From braids, $B_r$, on $r$ strings we get $H_r=B_r/\lrang{q_1\cdots q_{r\nm1}q_{r\nm1}\cdots q_1}$. \end{edesc}    

The case $r=4$ in \eql{Hrgens}{Hrgensa} is so important in examples, that in  {\sl reduced\/} Nielsen classes, we conveniently refer to $q_2$ as the {\sl middle twist}. As usual, in notation for free groups modulo relations,  \eql{Hrgens}{Hrgensb} means to mod out by the normal subgroup generated by the relation $q_1\cdots q_{r\nm1}q_{r\nm1}\cdots q_1=R_H$.

\begin{princ} From \eqref{Hrgens}, we get a permutation represention of $H_r$ on $\ni(G,\bfC)^\dagger$. Given $\dagger$, that gives a cover $\Phi\eqdef \Phi^{\dagger}: \sH(G,\bfC)^\dagger \to U_r$: The Hurwitz space of $\dagger$-equivalences of covers. 

The elements in $\lrang{q_1\cdots q_{r\nm1}q_{r\nm1}\cdots q_1}$ have the affect  \begin{equation} \label{inneraction} \begin{array}{c} \bg\in \ni(G,\bfC)\mapsto g\bg g^{-1}\text{ for some }g\in G. \text{ Indeed, for} \\ 
 \bg\in \ni(G,\bfC), \{(\bg)q^{-1}R_Hq\mid q\in B_r\}=\{g^{-1}\bg g\mid g\in G\}.\end{array} \end{equation}  \end{princ}

Denote the subgroup of the normalizer, $N_{S_n}(G)$, of $G$ in $S_n$ that permutes a given collection, $\bfC$, of conjugacy classes, by $N_{S_n}(G,\bfC)$.  Circumstances dictate when we identify  covers $f_u: {}_uW\to\prP^1_z$, $u=0,1$, branched at ${}_0\bz$, obtained from any one cover using the dragging-branch-points principle. 
Two equivalences that occur on the Nielsen classes:   
\begin{edesc} \label{eqname}  \item  \label{eqnamea} {\sl Inner}: $\ni(G,\bfC)^\inn\eqdef \ni(G,\bfC)/G$ corresponding to \eqref{inneraction}. \item  \label{eqnameb} {\sl Absolute}: Form $\ni(G,\bfC)^\abs\eqdef \ni(G,\bfC)/N_{S_n}(G,\bfC)$. 
\end{edesc} One might regard {\sl Inner\/} (resp.~{\sl Absolute}) equivalence as {\sl minimal\/} (resp.~{\sl maximal}). Act by $H_r$ on either  equivalence (denoted by a $\dagger$ superscript). 

\begin{defn}[Reduced action] \label{redaction} A  cover $f: W\to \prP^1_z$ is {\sl reduced\/} equivalent to $\alpha\circ f: W\to \prP^1_z$ for $\alpha\in \PSL_2(\bC)$. \end{defn} 

Also,  $\alpha$ acts  on $\bz\in U_r$ by acting on each entry. That extends to an action on any cover $\Phi^\dagger: \sH(G,\bfC)^\dagger\to U_r$, giving a reduced Hurwitz space cover: 
\begin{equation} \label{reducedcover}  \Phi^{\dagger,\rd}: \sH(G,\bfC)^{\dagger,\rd} \to U_r/\PSL_2(\bC)\eqdef J_r.\end{equation}

\subsubsection{Genus formula for $r=4$} \label{genus}  When $r=4$, $U_r/\PSL_2(\bC)$ identifies with $\prP^1_j\setminus\{\infty\}$. 
A reduced Hurwitz space of 4 branch point covers is a natural $j$-line cover.  That completes to $\overline \sH(G,\bfC)^{\dagger,\rd} \to \prP^1_j$ ramified over $0, 1, \infty$. Denote the group $\lrang{q_1q_3^{-1},\sh^2}$ by $\sQ''$. 

\cite[\S4.2]{BFr02} contains the formula whose statement in Thm.~\ref{genuscomp} uses   
$$\text{\sl reduced Nielsen classes\ } \ni(G,\bfC)^{\dagger,\rd}\eqdef \ni(G,\bfC)^\dagger/\sQ''. $$

\begin{thm} \label{genuscomp} Suppose a component, $\overline{\sH'}$, of $\overline \sH(G,\bfC)^{\dagger,\rd}$ is given by a braid orbit, $O$, on the corresponding Nielsen classes $\ni(G,\bfC)^{\dagger,\rd}$. Then, the ramification, respectively over $0,1,\infty$, of $\overline {\sH'}\to \prP^1_j$ is given by the disjoint cycles of $\gamma_0=q_1q_2$, $\gamma_1=q_1q_2q_1$, $\gamma_\infty=q_2$ acting on $O$. 

The genus, $g_{\overline {\sH'}}$, of $\overline{\sH'}$, a la Riemann-Hurwitz,  appears from  $$2(|O|+g_{\bar \sH'}-1)=\ind(\gamma_0)\np \ind(\gamma_1)\np \ind(\gamma_\infty).$$ \end{thm} 

\begin{rem} Notice that $\gamma_1\gamma_2\gamma_3=1$ (product-one) is a conseqence of the Hurwitz braid relation $$q_1q_2\cdots q_{r\nm1}q_{\nm1}q_{r\nm2}\cdots q_1=R_H$$ combined for $r=4$ with modding out by $q_1=q_3$. Also, that immediately gives $\gamma_0^3=1$ in its action on reduced Nielsen classes. Hint: Use also the braid relations $q_iq_{i\np1}q_i=q_{i\np1}q_iq_{i\np1}$. \end{rem} 

\begin{rem}[Braid orbits] \label{braidorbits1} From Thm.~\ref{genuscomp} it becomes clear quickly that identifying the braid orbits, $O$, in the Nielsen class is crucial. See, for example, Rem.~\ref{braidorbits2} on the \MT\ from the level 0 Nielsen class \eql{mainexs}{mainexsc} as the prime $\ell$ changes in \cite{FrH15} or \cite[\S 5]{Fr18}. \end{rem} 

We now do one example to illustrate aspects of the genus calculation for Thm.~\ref{genuscomp}. Especially we show the theory is in place to compute. \cite{BFr02} is our main source for the theory and other examples illustrating the theory, purposely chosen to show on one full example of a \MT, of what one might care about if these were modular curves, though they are not. Level 0 of the \MT\ for that example is designated $\ni(A_5,\bfC_{3^4})$. It  has just one braid orbit, unlike the example we now do which has two. 

This example happens to be level 0 for the prime $\ell=2$ of our example illustrating how to generalize Serre's \OIT\  \sdisplay{\cite{FrH15}}.  It is $\ni(A_4,\bfC_{+3^2-3^2})^{\inn,\rd}$, which has a rational union of conjugacy classes, thereby defining Hurwitz spaces over $\bQ$ from the \BCL\ (Thm.~\ref{bclthm}).  Before doing the example, I list what to expect from it. 

\begin{edesc} \label{A44} \item \label{A44a} The Hurwitz space has two components, labeled $\sH_0^{\pm}$, that we will see clearly using the $\sh$-incidence matrix \sdisplay{\cite{BFr02}}. 
\item \label{A44b} The Hurwitz spaces have fine moduli, but neither component has fine {\sl reduced\/} moduli (criterion of \cite[Prop.~4.7]{BFr02}). 
\item \label{A44c} The genus of both components is 0.   
\item \label{A44d} Neither of the components is a modular curve, but we can compute their arithmetic and geometric monodromy as  $j$-line covers. \end{edesc} 

Comment on \eql{A44}{A44b}: Fine moduli for inner Hurwitz spaces in this case comes from $A_4$ having no center. The check for fine moduli on a braid orbit $O$ for the reduced version has two steps \cite[\S4.3.1]{BFr02}: $\sQ''$ must act as a Klein 4-group (called b(irational)-fine moduli);  and neither $\gamma_0$ nor $\gamma_1$ has fixed points (on $O$). 

Comment on \eql{A44}{A44c}: To apply the argument of \sdisplay{\cite{Fr06}} to conclude Main \MT\ conj.~\ref{mainconj} for this \MT\ requires going to higher levels to assure the component genuses rise beyond 1. Each component has 2-cusps (labeled respectively $O^4_{1,1}$ and $O^4_{1,4}$). It requires more work to estabish the explicit rise of genus, but this is the crucial hypothesis of  \cite[\S5]{Fr06}. 

Subdivide $\mapsto \ni(A_3,\bfC_{\pm 3^2})^{\inn,\rd}$ according to the 
sequences of conjugacy classes $\C_{\pm 3}$;   $q_1q_3^{-1}$ and $\sh$ switch these rows: 
$$\begin{tabular}{cccc} &[1] +\,-\,+\,- &[2] +\,+\,-\,- &[3] +\,-\,-\,+ \\
&[4] -\,+\,-\,+ &[5] -\,-\,+\,+ &[6] -\,+\,+\,-
\end{tabular} $$

We limit the rest of this example to displaying the two components, and their genuses.  Here is the sh-incidence matrix notation for cusps, labeled $O_{i,j}^k$: $k$ is the cusp width, and $i,j$
corresponds to a labeling of orbit representatives. 
 
 Consider $g_{1,4}=((1\,2\,3), (1\,3\,4), (1\,2\,4), (1\,2\,4))$. Its $\gamma_\infty$ orbit $O_{1,4}^4$ is what Thm.~\ref{level0MT} calls double identity (repeated elements in positions 3 and 4). There are also two other double identity cusps with repeats in positions 2 and 3, denoted $O_{3,4}^1$ and $O_{3,5}^1$. The following elements are in a Harbater-Mumford component \eqref{HMrep}. 
 $$\begin{array}{rl} \text{H-M rep.} \mapsto \bg_{1,1}=& ((1\,2\,3),
(1\,3\,2), (1\,3\,4), (1\,4\,3))\\
\bg_{1,3}=& ((1\,2\,3), (1\,2\,4), (1\,4\,2), (1\,3\,2))\\
\text{H-M rep.} \mapsto  \bg_{3,1}=& ((1\,2\,3), (1\,3\,2), (1\,4\,3), (1\,3\,4))
\end{array} $$ 

\begin{table}[h] \label{sh-incA4}
\begin{tabular}{|c|ccc|}  \hline $\ni_0^+$ Orbit & $O_{1,1}^4$\ \vrule  &
$O_{1,3}^2$\ \vrule & 
$O_{3,1}^3$ \\ \hline $O_{1,1}^4$ 
&1&1&2\\ 
$O_{1,3}^2$ &1 &0&1  \\ $O_{3,1}^3$ &2&1&0 \\  \hline \end{tabular} 
\ 
\begin{tabular}{|c|ccc|} \hline $\ni_0^-$ Orbit & $O_{1,4}^4$\ \vrule  &
$O_{3,4}^1$\ \vrule & 
$O_{3,5}^1$ \\ \hline $O_{1,4}^4$ 
&2&1&1\\ 
$O_{3,4}^1$ &1 &0&0  \\ $O_{3,5}^1$ &1&0&0 \\  \hline \end{tabular} \end{table}

\begin{prop} \label{A43-2} On $\ni(\Spin_4,\bfC_{\pm3^2})^{\inn,\rd}$
(resp.~$\ni(A_4,\bfC_{\pm3^2})^{\inn,\rd}$)
 $H_4/\sQ''$ has one (resp.~two) orbit(s). So,  $\sH(\Spin_4,\bfC_{\pm3^2})^{\inn,\rd}$ 
 (resp.~\!$\sH(A_4,\bfC_{\pm3^2})^{\inn,\rd}$) has one (resp.~two)  component(s),  
 $\sH_{0,+}$ (resp.~$\sH_{0,+}$ and $\sH_{0,-}$). 

Then, $\sH(\Spin_4,\bfC_{\pm3^2})^{\inn,\rd}$ maps one-one to 
$\sH_{0,+}$ (though changing $A_4$ to $\Spin_4$ give different moduli).
The compactifications of $\sH_{0,\pm}$ both have genus 0 from Thm.~\ref{genuscomp} (Ex.~\ref{exA43-2}).\end{prop} 

The diagonal entries for 
$O_{1,1}^4$ and $O_{1,4}^4$ are nonzero. In detail, however, $\gamma_1$
(resp.~$\gamma_0$) fixes 1 (resp.~no) element of $O_{1,1}$, and neither of $\gamma_i$, $i=0,1$, fix any
element of $O_{1,4}^4$. 

\begin{exmp}[Computing the genus] \label{exA43-2} Use  $(\gamma_0,\gamma_1,\gamma_\infty)$ from the $\sh$-incidence
calculation in Prop.~\ref{A43-2}. Denote their restrictions to lifting
invariant $+1$  (resp.~-1) orbit  
 by $(\gamma_0^+,\gamma_1^+,\gamma_\infty^+)$ (resp.~$(\gamma_0^-,\gamma_1^-,\gamma_\infty^-)$).
We read indices of the $+$ (resp.~$-$) elements from the $\ni_0^+$ (resp.~$\ni_0^-$) matrix block: 
Cusp  widths over
$\infty$ add to the degree $4+2+3=9$ (resp.~$4+1+1=6$) to give
$\ind(\gamma_\infty^+)=6$ (resp.~$\ind(\gamma_\infty^+=3$); since $\gamma_1^+$ (resp.~$\gamma_1^-$) has 1
(resp.~no) fixed point and $\gamma_0^\pm$ have no fixed points, $\ind(\gamma_1^+)=4$
(resp.~$\ind(\gamma_1^-)=3$) and
$\ind(\gamma_0^+)=6$ (resp.~$\ind(\gamma_0^+)=4$). The genus of $\bar \sH_{0,\pm}$ is $g_{\pm}=0$:  
 $$2(9+g_+-1)=6+4+6=16\text{  and }2(6+{g_-}-1)=3+3+4=10.$$
\end{exmp}

\subsubsection{Defining \MT s} \label{modtowdef}  
Let $\psi_i:H_i\to G$, $i=1,2$, be Frattini covers (Def.~\ref{frattdef}). 

\begin{lem} \label{fiberprod} A minimal (not necessarily unique) subgroup $H\le H_1\times_G H_2$ that  is surjective to $G$, is a Frattini cover of $G$ that factors surjectively to each $H_i$. Thus, Frattini covers of $G$ form a projective system. From their definition, taking a Frattini cover of a group preserves the rank. \end{lem} 

\begin{proof} Consider the projection $\pr_i: H\le H_1\times_G H_2 \to H_i$, $i=1,2$.   Then, $\pr_i(H)$ is a subgroup of $H_i$ that maps surjectively to $G$. Since $\psi_i$ is a Frattini cover, $\pr_i(H)=H_i$, $i=1,2$. \end{proof} 

Also, Frattini covers of perfect groups are perfect.  Key for Frattini covers is that $\ker(\psi)$ is nilpotent \cite[Lem.~20.2]{FrJ86}${}_{1}$ or \cite[Lem.~22.1.2]{FrJ86}${}_{2}$. 

Write $\ker(\psi)=\prod_{\ell ||G|} \ker(\psi)_{\ell}$ indicating the product is over its $\ell$-Sylows. For each $\ell$, quotient by all Sylows for primes other than $\ell$ dividing $\ker(\psi)$. Thus form ${}_\ell\psi: {}_\ell H \to G$. 

The fiber product of the ${}_\ell H\,$s over $G$ equals $H$.  \cite[App.~B]{Fr18}  discusses elementary structural statements about the construction of $\tilde G$ below. Some version of these appear in \cite[Chap.~22]{FrJ86}${}_2$. 

\begin{defn} \label{univFrattDef} This produces a profinite cover, the {\sl Universal Frattini cover}, $\tilde \psi_G: \tilde G \to G$. Similarly,   $$\tilde \psi_{G,\ab}: \tilde G/ [ \ker( \tilde \psi_{G}), \ker( \tilde \psi_{G})] \eqdef \tilde G_\ab \to G$$ is the Universal  {\sl Abelianized\/} Frattini cover of $G$. \end{defn} 

Then, $\tilde \psi_G$ is a minimal projective object in the category of finite groups covering $G$.  So, given any profinite group cover $\psi: H \to G$, some homomorphism $\tilde \psi_{G,H}$ to $H$ factors through $\psi$. If $\psi$ is a Frattini cover, then $\tilde \psi_{G,H}$ must be a Frattini cover, too.

If $\rk(G)=t$, you may construct $\tilde G$ using a pro-free group, $\tilde F_t$, on the same (finite) number of generators. Send its generators to generators of $G$ to give a cover, $\tilde F_t\to G$. Then, define $\tilde G$ to be a minimal (closed) subgroup of $\tilde F_t$ covering $G$. 

This, though is nonconstructive.  It uses the Tychynoff Theorem: a nested sequence of closed subgroups of $\tilde F_t$ covering $G$ has non-empty intersection covering $G$.  That also explains why it wasn't sufficient to replace $\tilde F_t$ by the free (rather than pro-free) group on $t$ generators. 

The following quotients of $\tilde G$ are more accessible and extremely valuable for forming results and conjectures. Their existence follows  from decomposing the (pro-)nilpotent kernel $\ker(\tilde \psi)$ into a product of its $\ell$-Sylows. 

\begin{defn} \label{univFrattdef-ab} For each prime $\ell| |G|$, there is a profinite Frattini cover ${}_\ell \tilde \psi_{G}: {}_\ell \tilde G \to G$ with $\ker({}_\ell\tilde \psi_{G})$ a  profree pro-$\ell$ group of finite rank, $\rk(\fG \ell)$.  There are similar such covers with ${}_\ell\tilde \psi_{G,\ab}$ replacing ${}_\ell\tilde \psi_{G}$. \end{defn} 

The Frattini subgroup of an $\ell$-group $H$ is the (closed) subgroup generated by $\ell$-th powers and commutators from $H$. Denote it by ${}_{\text{\rm fr}}H$. 

Recover a cofinal family of finite quotients of $\fG \ell$ through the Frattini kernel of the natural map $1\to \ker_0 \to \fG \ell \to G\to 1$. This produces what we call the sequence of {\sl characteristic kernels\/} of $\fG \ell$: \begin{equation} \label{charlquots} \ker_0> {}_{\text{\rm fr}} \ker_0\eqdef  \ker_1 \ge \dots \ge  {}_{\text{\rm fr}}\ker_{k{-}1} \eqdef \ker_k \dots\end{equation}   Denote  $\fG \ell/\ker_k$ by $\tfG \ell k$, and $\ker_k/\ker_{k'}$ by ${}_\ell M_{k,k'}$ or  $M_{k,k'}$ for $k'\ge k$. \begin{equation} \label{charZlGmod} \text{Especially, ${}_\ell M_G\eqdef {}_\ell M_{0,1}$ is the characteristic $\bZ/\ell[G]$ module.} \end{equation}

Given $(G,\bfC,\ell)$, define the Nielsen classes ${}_\ell\ni(G,\bfC)$ of a modular tower in the following way that extends, in a profinite sense that of an ordinary Nielsen class (Def.~\ref{NielsenClass}). 

Denote the free group $\pi(U_{\bz_0}, z_0)$ modulo inner automorphisms, by $\sG_{\bz_0}$. Then, consider $\psi_{\bg}: \sG_{\bz_0} \to G$ given by mapping classical generators \eqref{generatorspi}, $\sP$,  to the branch cycles $\bg\in \ni(G,\bfC)$. 

Now form all possible homomorphisms $\psi_{\tilde \bg}: \sG_{\bz_0} \to \fG \ell$ through which $\psi_{\bg}$ factors, indicating images of the classical generators, $\sP$, by $\tilde \bg$, that satisfy this additional condition: 
\begin{equation} \text{$\tilde g_i$ has the same order as $g_i$, $1,\dots, r$.} \end{equation} 
From Schur-Zassenhaus,  as $\fG \ell\to G$ has kernel an $\ell$-group, this defines the conjugacy class of $\tilde g_i$ uniquely. With no loss, also label it $\C_i$, $i=1,\dots,r$. 

This makes sense of writing $\ni(\fG \ell,\bfC)^\dagger$ (or $\ni(\fG \ell {}_\ab,\bfC)^\dagger$) with $\dagger$ any one of the equivalences we have already discussed in \S\ref{equivalences}. As previously $H_r$ acts on the Nielsen classes. To define the Nielsen class levels, mod out successively, as in Def.~\ref{univFrattDef}, on $\fG \ell$ by the characteristic kernels of \eqref{charlquots}. Then, $H_r$ acts on all of these Nielsen classes. 

\begin{defn}[\MT\ Def] \label{MTdef} Assume $O$ is an  $H_r$ orbit on $\ni(\fG \ell,\bfC)^\dagger$ (resp.~$\ni(\fG \ell {}_\ab,\bfC)^\dagger$). Then, the level $k$ Nielsen class is the image orbit ${}^kO$ of $O$ in $\ni(\tfG \ell k, \bfC)^\dagger$ (resp.~$\ni(\tfG \ell k {}_\ab,\bfC)^\dagger$), with its corresponding Hurwitz space.  \end{defn} 

\section{The path to {\bf M}(odular){\bf T}(ower)s}  \label{overview} \S\ref{S1title} elaborates on \lq\lq What Gauss told \dots.\rq\rq\ a hidden history that has obscured the nonabelian aspects of {\bf R}(iemann)'s{\bf E}(xistence){\bf T}(heorem). \S\ref{Fr18abstract} is an abstract for \cite{Fr18}. While the universal Frattini cover of any finite group $G$ allows launching into such non-abelian aspects should you dare, \cite{Fr18} takes a middle road. 

The applications that extend the \OIT\ and feature the $\ell$-adic monodromy come at the end of the book. The applications that feature the \RIGP\ and the interpretation of that problem as interpreted by rational points on Hurwitz spaces come at the beginning of the book. The middle of the book joins those two areas, by emphasizing the universal Frattini cover $\tilde G$, and braid orbits on Nielsen classes.  

Suppose, you accede to taking on a serious simple group (say, $A_5$ and the prime $\ell=2$), and a very small $\ell$-Frattini cover. For example, as in \cite[Chap.~9]{Se92},  the short exact sequence $0 \to \bZ/2 \to \Spin_5\to A_5\to 1$. Then you might want to recognize the sequence for the abelianized $2$-Frattini cover \begin{equation} \label{A52}  0\to (\bZ_2)^5\to {}_2\tilde A_{5,\ab}\to A_5\to 1 \text{ \cite[Prop.~2.4]{Fr95}} \end{equation} (and its characteristic quotients) as quite a challenge at the present time,\footnote{ $\Spin_5\to  A_5$ is the smallest nontrivial quotient of ${}_2^1 A_5\to A_5$, as in \eqref{charlquots}.} even according to Conj~\ref{MTconj} if you only had to find {\sl any\/} number field $K$ for which all those characteristic quotients have \RIGP\ realizations over $K$.

We emphasize the value of the case  $r=4$. This plays on modular curve virtues, as being upper half-plane quotients. This, although congruence subgroups do not define the \MT\ levels except when $G$ is close to dihedral. 

\subsection{Guiding examples} \label{guides} 
\S\ref{explainingFrattini} explains the logic of the book, with extended abstracts on several papers on \MT s, starting with the first in 1995. The theme is what I learned from four examples. 

\begin{edesc} \label{mainexs} \item \label{mainexsa} Abel's spaces  as level 0 of a \MT\ classically denoted $ \{X_0(\ell^{k\np1})\}_{k=0}^\infty$. 
\item \label{mainexsb} The \MT\ from the Nielsen class $\ni(A_5,\bfC_{3^4})$ and the prime $\ell=2$. 
\item \label{mainexsc} The \MT\ system, from the Nielsen classes $\ni((\bZ/\ell)^2\xs\,\bZ/3, \bfC_{3^4})$, running over primes $\ell$, as my foray into an \OIT\ beyond Serre's.\footnote{Precisely what to do about $\ell=3$, is tricky. (See \sdisplay{\cite{FrH15}} starting with \eqref{3vs2}.)} 
\item \label{mainexsd} Relating $ \{X_1(\ell^{k\np1})\}_{k=0}^\infty$ and  $ \{X_0(\ell^{k\np1})\}_{k=0}^\infty$ is a special case of a general relation between {\sl inner\/} and {\sl absolute\/} Hurwitz spaces. 
\end{edesc} 

In \eql{mainexs}{mainexsa} and \eql{mainexs}{mainexsd}, there is a parameter $k$ indicating a tower level. Since in these two cases, the tower levels are traditionally related to a power of a prime $\ell$, I have assumed no more explanation of $k$ as a level is required. 

The inbetween cases, \eql{mainexs}{mainexsb} and \eql{mainexs}{mainexsc}, especially in the former, the levels would be less obvious. Yet, as with all \MT s, there is a prime $\ell$ and a level $k$, by which $\ell^{k\np1}$, $k\ge 0$ is naturally attached to that level. 

My initial relation with the \OIT, during my first decade as a mathematician, was based on my interpretation of \eql{mainexs}{mainexsa} as a moduli space. This went through several stages on very practical problems with considerable literature, on which this paper elaborates. The two series in \eql{mainexs}{mainexsd} are reasonably considered the mainstays of modular curves. 

My interactions with Serre on \cite{Se90a} and \cite{Se90b}, before they were written, and related to my review of \cite{Se92} (see \cite{Fr94})  before it appeared  caused me to go deeply into \eql{mainexs}{mainexsb}. Indeed, it was \cite{Fr90} that Serre saw me present in Paris in 1988, that had him write to me asking -- essentially -- for the formula for the lift invariant for the families of genus 0 covers whose Nielsen classes are given as $\ni(A_n,\bfC_{3^{n-1}})$, $n\ge 4$. 

For that reason, I have alluded to \cite{Fr12} in the discussions. Especially, applied to a {\sl braid orbit\/}  on a Nielsen class, 
$$\text{the idea of the lift (braid) invariant  from a {\sl central  Frattini cover}.} $$  The comparison between general and central Frattini covers of a finite group appears in many places to interpret \MT s.   

I aimed with \eql{mainexs}{mainexsc} to show commonalities and differences from the source of Serre's \OIT\ \eql{mainexs}{mainexsa}. Especially I refer to what works for modular curves, and what I learned that applies even to modular curves, though not previously observed, or the approach is different/illuminating. That is the concluding topic of \cite{Fr18}. Therefore I am brief on it here, merely recording some of its results that show what is new from an example that goes beyond Serre's \OIT. 

\subsection{What Gauss told Riemann about Abel's Theorem}  \label{S1title}  The title is the same as that of the paper \cite{Fr02}.  Since, among all upper half-plane quotients so many study modular curves, it must be shocking that most upper half-plane quotients -- $j$-line covers ramified over $0,1,\infty$ -- are not modular curves. One lesson from ${}_\ell \sX_0\eqdef \{X_0(\ell^{k\np1})\}_{k=0}^\infty$  came from Galois. 

He computed the geometric monodromy of $X_0(\ell)\to \prP^1_j$ (over $\bC$; $k=0$). finding it to be $\PSL_2(\bZ/\ell)$ which is simple for $\ell\ge 5$. As an example of his famous theorem:  {\sl radicals\/} don't generate the algebraic functions describing modular curve covers of the $j$-line. 

\subsubsection{An abstract  for \cite{Fr18}} \label{Fr18abstract}

\begin{center} {\large Monodromy, $\ell$-adic representations \\ and the  Inverse Galois Problem}   \end{center} 

This book connects to, and extends, key themes in two of Serre's books:

\begin{edesc} \label{serreBooks} \item  \label{serreBooksa} \emph{Topics in Galois theory} (the original and enhanced reviews \cite{Fr94}, and in French, translated by Pierre D\`ebes \cite{D95}); and
\item \label{serreBooksb} \emph{Abelian $\ell$-adic representations and Elliptic Curves\/} (see the 1990 review by Ken Ribet \cite{Ri90}). \end{edesc} 
It illustrates a general relation between $\ell$-adic representations, as in generalizing Serre's {\bf O}(pen){\bf I}(mage){\bf T}(heorem), and the {\sl Regular version\/} of the Inverse Galois Problem. These come together over the {\bf M}(odular){\bf T}(ower) generalization of modular curves. 

\cite[\S 4]{Fr18} explains MTs (started in 1995) as a program motivated by such a relationship. \cite[ \S 1 and \S 2]{Fr18} includes exposition that ties up threads of work coming just before \MT s. Especially, it recasts \cite{FrV91} and \cite{FrV92} to  modernize investigating definition fields of components of Hurwitz spaces vis-a-vis lift invariants with examples. 

Then, we interpret the \OIT\ in a generality not indicated by Serre's approach. The discussion \sdisplay{Fr15}, connects to the technical paper behind this material.   \cite[\S 5]{Fr18} uses one example -- in  that all modular curves are one example -- that is clearly not of modular curves. This explains why our approach to (families of) covers of the projective line can handle a barrier noted by Grothendieck to generalizing the \OIT. \cite[\S 3]{Fr18} joins the 1st and last 3rd of the book, in a new approach to the {\sl lift invariant \/} and Hurwitz space components based on the {\sl Universal Frattini cover\/} of a finite group. 

\sdisplay{\cite{Se90a}} explains the bifurcation point in \MT s as it applies to any particular finite group $G$.   
\begin{edesc} \label{bifurcation} \item \label{bifurcationa} Using the {\sl lift invariant\/} to \lq\lq freely\rq\rq\ construct Hurwitz spaces over $\bQ$ for some $G^*$ covering $G$  with one connected component; vs 
\item \label{bifurcationb} Using the universal Frattini cover $\tilde G\to G$  to construct towers of spaces based on $G$ that generalize modular curve towers. \end{edesc} 

A sense of the difference between \eql{bifurcation}{bifurcationa} and \eql{bifurcation}{bifurcationb} for $G=A_5$ appears in the difference between $\Spin_5\to A_5$ and the whole sequence \eqref{A52}. 

\subsubsection{Early  \OIT\ and \MT s relation} \label{earlyOIT-MT}  For any field $K$, denote the $K$ points on a (quasi-projective) algebraic variety $V$ by $V(K)$. 

The \MT\ description of $X_0(\ell)$ in \eql{mainexs}{mainexsa} shows they are the absolute reduced Hurwitz spaces $\sH(D_{\ell^{k\np1}},\bfC_{2^4})^{\abs,\rd}$ with $D_{\ell^{k\np1}}$ the {\sl dihedral group\/} of order $2\cdot \ell^{k\np1}$, $\ell$ odd, and $\bfC_{2^4}$ four repetitions of the involution (order 2) conjugacy class. There were two stages in this recognition. 

{\sl Stage 1}: \cite[\S2]{Fr78} starts the observation of the close relation between Serre's \OIT\ and the description of {\sl Schur covers\/} $f\in \bQ(w)$: \begin{equation} \label{schurcover} \!\!\!\text{$f: \prP^1_w(\bZ/p) \to \prP^1_z(\bZ/p)$ is one-one for $\infty$-ly many $p$.}\end{equation}  Similarly over any number field $K$ replacing $\bQ$, $f\in K(w)$, and residue class fields $\sO_K/\bp=\bF_\bp$ replacing $p$: 
\begin{equation} \label{schurcovernf} \!\!\!\text{$f: \prP^1_w(\sO_K/\bp) \to \prP^1_z(\sO_K/\bp)$ is one-one for $\infty$-ly many $\bF_\bp$.}\end{equation} 
The Galois closure of such an $f$ may only be defined over a proper extension $\hat K/K$. Indeed,  for $f$ to be a Schur cover over $K$, we must have $\hat K\not= K$.   

\newcommand{\exc}{{\text{\rm Exc}}}
Use the permutation notation of Rem.~\ref{permnot},  $T_f: \hat G_f\to S_n$ and the respective stabilizers of 1 by $\hat G_f(1)$, $G_f(1)$. 

To prevent accidents, define the exceptional set (for the Schur property): $$\exc_{f,K}=\{\bp \mid f \text{ is one-one on $\infty$-ly many extensions of }\bF_\bp\}.$$

If $f_i$, $i=1,2$, are exceptional over $K$, then  \begin{equation} \label{capexc} \text{so is $f_1\circ f_2$, if $|\exc_{f_1,K}\cap \exc_{f_2,K}| = \infty$.}\end{equation} That is, $f_1\circ f_2=f$ is a decomposition of $f$ over $K$.  The condition that $f$ is {\sl indecomposable\/}  over $K$ is that $T_f: \hat G_f\to S_n$ is {\sl primitive}: There is no group $H$ properly between $\hat G_f(1) $ and  $\hat G_f$.

\begin{prob}[Schur Covering]  Find all {\sl exceptional \/} $f$ {\sl indecomposable\/} over $K$. \end{prob}  

What made it possible to see the connection of exceptional $f$ with the \OIT, is that Thm.~\ref{schurcoverfiber} gives a complete characterization of exceptionality based on the relation between the arithmetic and geometric monodromy of the covers given by $f$.  

\begin{thm} \label{schurcoverfiber} The following is equivalent to exceptionality. 

There exists $\hat g \in \hat G_f(1)$, such that each orbit of $\lrang{G_f(1),\hat g}$ on $\{2,\dots,n\}$   breaks into (strictly) smaller orbits under $G_f(1)$ \cite[Prop.~2.1]{Fr78}. \end{thm}  

Thm.~\ref{schurcoverfiber} holds for essentially any cover (absolutely irreducible over $K$; the  sphere need not be the domain)  \cite[Prop.~2.3]{Fr05}. This application of a wide ranging Chebotarev density theorem is a case of  {\sl monodromy precision\/}. Usually a {\sl Chebotarev density\/} application in the $\implies$ direction isn't so precise. Yet, here instead of saying $f$ is {\sl almost\/} one-one, it implies it is exactly one-one for $\infty$-ly many $\bp$. 

\cite{Fr05} expanded on situations giving monodromy precision. This is as an improvement on the intricate, and somewhat endless process of refining the appropriate error term in the Riemann hypothesis over finite fields. 

To relate to the \OIT\ for rational $f$, it turned out enough to concentrate on two cases with $\ell$ prime: $\deg(f)=\ell$, or $\ell^2$. The next step was to describe those Nielsen classes that produce the exceptional $f$. \cite[Thm.~2.1]{Fr78} lists the Nielsen classes (of genus 0 covers) for $\deg(f)=\ell$ that satisfy these conditions, noting a short list of 3-branch point cases, with a main case of $r=4$ branch points forming one connected family.  

Thereby, it identifies $X_0(\ell^{k\np1})$ (resp.~$X_1(\ell^{k\np1})$)  as natural reduced absolute (resp.~inner) Hurwitz spaces as in \eql{mainexs}{mainexsd}. This was a special case of \cite[\S3]{Fr78}, the {\sl extension of constants\/} rubric for covers, by going to their Galois closure. See the additional discussion under \sdisplay{\cite{Fr78}}. The moduli interpretation is incoded in the  cover sequence  $$\sH(G,\bfC)^\inn\to \sH(G,\bfC)^{\abs} \to U_r$$  and its expansion into total spaces (over $U_r\times \prP^1_z$) as exploited in \cite[Thm.~1]{FrV91}.  This codified -- in moduli terms -- the relation between $G_f$ (geometric) and $\hat G_f$ (arithmetic) monodromy in the language of \S\ref{whyrationalfuncts}. \cite[\S2]{Fr78} shows, for prime degree $\ell$ rational functions, identifying Schur $f$ covers is essentially equivalent to the theory of complex multiplication. Further, from that theory, we may describe $\exc_{f,K}$ as an explicit union of arithmetic progressions, thereby allowing testing the condition \eqref{capexc}.  

Describing prime-squared degree exceptional rational functions interprets the $\GL_2$ part of Serre's \OIT, as in \cite[\S6.1--\S6.3]{Fr05}. We state the main point, over $\bQ$, again with Chebotarev precision. This also fits the inner-absolute Hurwitz space relation above by using a different Nielsen class $\ni((\bZ/\ell)^2\xs \bZ/2,\bfC_{2^4})$, still a modular curve case. The Nielsen class collection $\{\ni((\bZ/\ell^{k\np1})^2\xs \bZ/2,\bfC_{2^4})\}_{k=0}^\infty$ gives its corresponding modular curve tower \cite[Prop.~6.6]{Fr05}. 

For a given $\ell$, from Serre's (eventual, Thm.\ref{OITsf}) version of the \OIT\, we conclude this. If the elliptic curve $E$ (say, over $\bQ$) has a $\GL_2$ $j$-invariant, $j_E=j_0$, then the corresponding degree $\ell^2$ rational function $f_{j_0,\ell^2}$ has arithmetic/geometric monodromy group quotient $\GL_2(\bZ/\ell)/\{\pm1\}=G(\bQ_{j',\ell^2}/\bQ)$ for all primes $\ell \ge \ell_0$ for some $\ell_0$ dependent on $j_0$ \cite{Se97b}. Denote the Frobenius (conjugacy class) at the prime $p$ in a Galois extension $\hat L/\bQ$ by $\Fr_{L,p}$. 

\begin{prop} \label{excGL2} Given any such $\ell\ge \ell_0$ as above, $$ \exc_{f_{j',\ell^2}}= \{p \mid \text{ $\lrang{-1,\Fr_{\bQ_{j',\ell^2},p}}$ acts irreducibly on $(\bZ/\ell)^2$\}.}$$   This is always infinite.  \end{prop} 

\begin{proof} The classical Chebotarev density theorem implies $\exc_{f_{j',\ell^2}}$ is infinite if any element of $\GL_2(\bZ/\ell)$ acts irreducibly on $(\bZ/\ell)^2$.  For example, on the degree 2 extension $\bF_{\ell^2}$ of $\bZ/\ell=\bF_\ell$, multiply by a primitive generator $\alpha$ of $\bF_{\ell^2}/\bF_\ell$ to get an invertible $2\times 2$ matrix with no invariant subspace. \end{proof} 

The $\GL_2$ case is vastly different from the $\CM$ case in that the exceptional set described in Prop.~\ref{excGL2} is definitely not a union of arithmetic progressions. \cite[\S6.3.2]{Fr05} relates to \cite{Se81} on using the (conjectural) Langlands program to consider these exceptional sets. 

Guralnick-M\"uller-Saxl \cite{GMS03} show that -- excluding those above -- other  indecomposable Schur covers by rational functions,  are {\sl sporadic\/}. That is,  they correspond to points on a finite set of Hurwitz space components.

\begin{defn}[Named Nielsen Classes] \label{namedclasses} Use the respective names \CM\ and $\GL_2$ for the Nielsen classes $\ni(D_\ell,\bfC_{2^4})^{\dagger,\rd}$ and $\ni((\bZ/\ell)^2\xs \bZ/2,\bfC_{2^4})^{\dagger,\rd}$ (or to the whole series with $\ell^{k\np1}, k\ge 0$ in place of $\ell$), with $\dagger$ referring to either inner or absolute equivalence, and the relation between them. \end{defn} 

{\sl Stage 2\/} discussion, and its relation to the \OIT, couldn't happen until there was a full formulation of \MT s, the topic of \S\ref{explainingFrattini}. Still,  a transitional phase after Stage 1 occurred with dihedral groups and the space of hyperelliptic jacobians \S\ref{Frat-Groth}, as discussed in \sdisplay{\cite{DFr94}}.  

A simple -- \MT\ free -- question there shows how \MT s for each finite group $G$ and $\ell$-perfect prime of $G$ arise automatically. This generalizes what the same question applied to dihedral groups poses for hyerelliptic jacobians. We consider this the simplest approach to how much the \MT\ project, though the \RIGP, connects to classical considerations. 

\subsubsection{Competition between algebra and analysis} \label{algvsanal} The full title of \cite[\S7]{Fr02} is {\sl Competition between algebraic and analytic approaches}. This subsection consists of brief extracts from it and \cite[\S10]{Fr02}. \S7 was gleaned partly from \cite{Ne81}, and my own observations based on \cite{Ahl79} and \cite{Sp57}. \S10 was a \lq\lq modern\rq\rq\ personal experience, telling about the world's \lq\lq appreciation\rq\rq\ of mathematical genius. 

{\sl Riemann's early education} \cite[\S7.1]{Fr02}:  Riemann was suitable, as no other German mathematician then was to effect the first synthesis of the \lq\lq French\rq\rq\ and \lq\lq German\rq\rq\ approaches in general complex function theory. 

{\sl Competition between Riemann and Weierstrass} \cite[\S7.2]{Fr02} and \cite[p.~93]{Ne81}:   In 1856 the competition between Riemann and Weierstrass became intense, around the solution of the {\sl Jacobi Inversion problem}. Weierstrass consequently withdrew the 3rd installment of his investigations, which he had in the meantime finished and submitted to the Berlin Academy. 

{\sl Soon after Riemann died} \cite[\S7.3]{Fr02} and \cite[p.~96]{Ne81}:  
After Riemann's death, Weierstrass attacked his methods often and even openly. Curiously, the only reference in Ahlfor's book -- to Riemann's use of {\sl Dirichlet's Principle\/} for constructing the universal covering space of a Riemann surface -- is this \cite[footnote on p.~121]{Ahl79}: 

\begin{quote} Without use of integration R.L.Plunkett proved the continuity of the derivative (BAMS65,1959). E.H.~Connell  and P.~Porcelli proved existence of all derivatives (BAMS 67, 1961). Both proofs lean on a topological theorem due to G. T. Whyburn. \end{quote} 
 
There is a complication in analyzing Neuenschwanden's thesis that this resulted in mathematicians accepting Riemann's methods. How would one document that this event resurrected the esteem of Riemann's geometric/analytic view? 

{\sl Final anecdote} \cite[\S 10]{Fr02}: While at the Planck Institute in Bonn, to give talks in the early 21st Century, I visited Martina Mittag, a humanities scholar, who had earlier visited UC Irvine. In private conversation she railed that mathematicians lacked the imagination of humanities scholars. Yet, she was vehement on the virtues of Einstein. 

I explained that Einstein was far from without precedent; that we mathematicians had geniuses with at least his imaginative. My example was Riemann: I called him the man who formed the equation that gave Einstein his scalar curvature criterion for gravity: his thesis, and admittedly not my expertise. \lq\lq Mike,\rq\rq\ she said, \lq\lq You're just making that up! Who is Riemann?\rq\rq  

I took the {\sl R\/} book in her (German) encyclopedia series from the shelves on her walls, without the slightest idea of what I would find. Opening to Riemann, I found this in the first paragraph: 

\begin{quote} Bernhard Riemann was one of the most profound geniuses of modern times. Notable among his discoveries were the equations that Einstein later applied to general relativity theory. \end{quote}

In modern parlance, what Gauss explained to Riemann was what -- when I was young -- were called the {\sl cuts\/}. These are always displayed with pictures that are impossible -- not hard, but rather {\sl cannot\/} exist, as explained on \cite[Chap.~4, \S 2.4]{Fr80} under the title \lq\lq Cuts and Impossible Pictures.\rq\rq 

The pictures are usually on covers that are cyclic, degree 2 or 3, as in \cite[p.~243]{Con78}, which, though, is  excellent in many ways for students not comfortable with algebra. What Riemann learned, again in modern parlance, is that you don't need -- explicitly -- the universal covering space, nor a subgroup of  its automorphism group, to produce  covers.  

\subsubsection{Profinite: Frattini and Grothendieck} \label{Frat-Groth} 
I gleaned \S\ref{algvsanal} from reading long ago (from \cite{Sp57}),  that Riemann's $\theta\,$s, in a sense defined Torelli space, the period matrix cover of the moduli space, $\sM_g$, of curves of genus $g$. This codifies the integrals that Riemann used to introduce one version of moduli of curves. \cite{Sp57}  does explain fundamental groups. Yet, it always relies on universal covering spaces. 

One famous theorem is that the universal covering space of $\sM_g$ is a (simply-connected, Teichm\"uller) ball: It is contractible, and many beautiful pictures are made from this in defining fundamental domains. 

The approach of \cite{Fr18}, using Hurwitz spaces, is based on the {\sl Universal Frattini\/} cover $\tilde G\to G$, and its abelianized version $\tilde G_\ab \to G$.  The relation between $\sM_g$ and Hurwitz spaces starts by realizing that the latter generalizes the former, in a sense that adding data divides the former into smaller pieces. Using that division effectively does not require you must know all finite groups (or even all simple groups). 

\cite[App.~B]{Fr18} has a section on how each problem appears to have its own appropriate finite groups, based on a well-known paradigm -- {\sl the Genus 0 problem\/}. This is commentary on \cite{Fr05} which is a guide inspired by the solution of problems like {\sl Schur's}, {\sl Davenport's\/} and {\sl Schinzel's}, from the middle of the 20th century or before. 

\cite{Fr18}, though, makes a more direct comparison between the two theories, based on using profinite groups and their finite quotients, that  define -- as they do with \MT s --  natural moduli spaces whose properties you can compute. Except in special cases, though, there are no open complex spaces which one may regard as universal covers. 

There are infinitely many spaces $\sM_\geng$ (resp.~$\sH(G,\bfC,\ell)$ indexed by $\geng$ (resp.~$(G,\bfC,\ell)$)  in each case. While the indexing seems more complicated in the Hurwitz case, even in the former case one's instinct is that they should all fit together. Much work does introduce techniques that work uniformly for all $\geng$. Grothendieck's famous {\sl Teichm\"uller\/} group attempted to gather their presence together into one profinite group with an hypothesis that he was describing $G_\bQ$. 

That created quite an industry. Still, \cite{FrV92} showed that the Hurwitz space approach was up to the challenge of describing properties of $G_\bQ$ that most mathematicians can understand. For example \eqref{FrVSns}.  

\begin{thm} \label{FrV91} We may choose a(n infinite) Galois algebraic extension  $L/\bQ$  so that  $G_\bQ$ has a presentation (see also Conj.~\ref{FrVConj}): 
\begin{equation} \label{FrVSns} 1\to  F_{\omega}=G_L \to G_\bQ \to \Pi_{n=2}^\infty  S_n \eqdef \sS_\infty \to 1 \end{equation} 
That is $G_\bQ$, has a product of $S_n\,$s as a quotient (the Galois group of $L/\bQ$) with the kernel a pro-free group on a countable set of generators. 
\end{thm} 

This exposition aims to show that \MT s, referenced by $(G,\bfC,\ell)$ are appropriate objects to  progress on problems that generalize those considered for modular curve towers. We illustrate that with \MT s of 4 branch point Nielsen classes. By example we now know that subtrees -- called {\sl spires\/} -- of the full tree of cusps are isomorphic to the tree of cusps on a modular curve tower, even when the group $G$ is nothing like a dihedral group.  

Now take any one of the extensions ${}_\ell \tilde G_\ab \to G$ each of whose characteristic quotients $\fG \ell/\ker_{k,\ab}$ -- as in Def.~\ref{univFrattdef-ab} -- represent the universal extension of $G$ with abelian exponent $\ell^{k}$ kernel. 

Start with $G=D_\ell$, $\ell$ odd. Assume there is $B>0$ and a $\bQ$ regular realization of $D_{\ell^{k\np1}}$ for each $k\ge 0$ with no more than  $B$ branch points. 

\begin{thm} \label{Dihmain} Then,  there exists $d< \frac{B\nm 2} 2$ and an $\ell^{k\np1}$ cyclotomic point (see below) on a hyperelliptic jacobian (varying with $k$) of dimension $d$, $k\ge 0$ \cite[\S5.2]{DFr94}. \end{thm} 

The proof of Thm.~\ref{Dihmain} goes through -- under the hypothesis -- proving the existence of a \MT\ for $(D_{\ell},\bfC_{2s})$ with $\bfC_{2^{2s}}$ ($2s=r$) repetitions of the involution class of $D_\ell$ for which each level of the \MT\ has a $\bQ$ point. Then, it gives the hyperelliptic jacobian interpretation. Assume $\bp$ is an $\ell^{k\np1}$ order torsion point on an abelian variety $A$ defined over $\bQ$. 

We say $\bp$ it is an $\ell^{k\np1}$ {\sl cyclotomic point\/} (on $A$) if $\lrang{\bp}$ is invariant under $G_\bQ$, and if $G_\bQ$ act on $\lrang{\bp}$ as it does on $\lrang{e^{2\pi i/\ell^{k\np1}}}$. 
 
\begin{guess} \label{Tormain} Torsion Conjecture: The conclusion of Thm.~\ref{Dihmain}, that there can be such a $\bQ$ cyclotomic point, for each $k$  on a hyperelliptic jacobian of a fixed dimension $d$ is false. 
  
$B$-free Conjecture: Don't stipulate any $B$. Conjecture: For each $\ell^{k\np1}$ there is a $\bQ$ cyclotomic point on some hyperelliptic jacobian, corresponding to a $(D_{\ell^{k\np1}},\bfC_{2^r})$ ($r$ dependent on $\ell^{k\np1}$) \RIGP\ involution realization.  \end{guess} 

Despite the last part of Conj.~\ref{Tormain}, no one has found those \RIGP\ {\sl involution} realizations beyond $r=4$ and $\ell=7$. The theme of \cite[\S7]{Fr94} -- using this paper's notation -- still seems reasonable. For any prime $\ell\ge 3$, as in \S\ref{earlyOIT-MT}, and given a choice, you should rather 
$$\text{regularly realize the {\sl Monster\/} than the collection $\{D_{\ell^{k\np1}}\}_{k=0}^\infty$,}$$ referring to the famous {\sl Monster\/} simple group. 

Now consider the analog for general $\ell$-perfect $G$. For example the $A_5$, $\ell=2$ case of \eqref{A52}. \sdisplay{\cite{FrK97}} has the documentation on this. 

\begin{thm} Suppose $B > 0$ and there exists a $\bQ$ regular realization with $\le B$ branch points  of ${}_\ell^{k\np1}G  = \tilde G_\ab/\ell^{k\np1}\ker(\psi_{G,\ell})$ for each $k$. 

Then,  there is a Nielsen class $\ni(G,\bfC)$, with $\bfC$ consisting of $r < B$ conjugacy classes, all $\ell'$,  and  a \MT\ $\{\sH_k\}_{k=0}^\infty$, 
\begin{equation} \label{MTconc} \text{ with $\sH_k$ a component of $\sH({}_\ell^{k} G,\bfC)^{\inn,\rd}$ and $\sH_k(\bQ)\not= \emptyset, k\ge 0$.}\end{equation}  \end{thm} 

\begin{guess} \label{mainconj} High {\bf MT} levels have {\sl general type\/} and no $\bQ$ points. \end{guess} 

\begin{thm}  \label{truer=4} Conj.~\ref{mainconj} is true for $r=4$, where $\sH_k\,$s are upper half-plane quotients. Thm.~\ref{genuscomp} is a tool for showing the genus rises with $k$. \end{thm}  

For $K$ a number field, concluding in Conj.~\ref{mainconj} that high \MT\ levels have no $\bQ$ points is of significance only if there is a uniform bound on the definition fields of the \MT\ levels. Therefore distinguishing between towers with such a uniform bound, and figuring the definition field as the levels grow if there is no uniform bound, is a major problem. 

Our approach allows us to compute, and to list properties of \MT\ levels. This is progress in meeting Grothendieck's objection that correspondences on jacobians can cause great complications in generalizing Serre's \OIT. 

In our \cite[\S5]{Fr18} example, that complication is measured by the appearance of distinct Hurwitz space components. The {\sl lift invariant\/} accounts for most. Still, others pose a particular problem at this time -- we know them, but not their definition fields, as {\sl Harbater-Mumford\/} components. 

That problem occurs because there is more than one with the same 0 lift invariant as discussed around Prop.~\ref{A43-2} and Thm.~\ref{level0MT}.  As in Thm.~\ref{genuscomp}, we know their braid orbits on the Nielsen class, though modular curves and complex multiplication are not a guide.

\subsection{The \TL\ of the \MT\ program} \label{explainingFrattini}  After a prelude we have divided this section into three subsections: 

\begin{itemize} \item \S\ref{pre95}, prior to 1995; 
\item  \S\ref{95-04}, the next decade of constructions/main conjectures, then,  
\item  \S\ref{05-to-now} of progress on the main \MT\ conjectures and our approach to the \OIT. \end{itemize} 

Conj.~\ref{OITgen} concludes the section with our \MT\ formulation of the \OIT. 

\subsubsection{Organization} \label{organization} This \TL \ picks out the gist of the main papers. Each item connects to a fuller explanation of the history and significance of the contribution. Three html files provide handy reminders on basics guiding progress on {\bf M}(odular) {\bf T}(ower)s. We refer to sections in them. 

The {\bf R}(egular) {\bf I}(nverse) {\bf G}(alois) {\bf P}(roblem), its literature and how  Nielsen classes (Def.~\ref{NielsenClass}) relate to the \MT\ conjectures: $$\text{http://www.math.uci.edu/$\tilde{\ }$mfried/deflist-cov/RIGP.html.}$$ 

Nielsen classes are a genus generalization that separates sphere covers into recognizable types. We use the previous notation: $\ni(G,\bfC)$ for (unordered) conjugacy classes $\bfC = \{\row \C r\}$ of a finite group $G$. {\bf R}(iemann)-{\bf H}(urwitz) \eqref{RH} gives the corresponding sphere cover genus $\geng\eqdef \geng_\bg$, if $(\row g r)\in \ni(G,\bfC)$. We have examples in $$\text{http://www.math.uci.edu/$\tilde{\ }$mfried/deflist-cov/Nielsen-Classes.html.}$$ 
 
The {\bf B}(ranch) {\bf C}(ycle) {\bf L}(emma) ties definition fields of covers (and their automorphisms) to branch point locations: $$\text{http://www.math.uci.edu/$\tilde{\ }$mfried/deflist-cov/Branch-Cycle-Lem.html.}$$ Especially, it gives the precise definition field of Hurwitz families defined by Nielsen classes $\ni(G,\bfC)$ for any equivalences (as in \S\ref{equivalences}). While this result is key for number theory results (on the \RIGP, and generalizing Serre's \OIT), we emphasize just one easily stated corollary. 

\begin{thm} \label{bclthm} The total space of an inner Hurwitz space, $\sH(G,\bfC)^{\inn}$ together with its extra structure as a moduli space over (even as a reduced Hurwitz space), of $\prP^1_z$ covers,  is a cyclotomic field given in the response to \eql{mtneed}{mtneedc} as \eqref{bcl}. In particular, an inner Hurwitz (moduli) space  structure is defined over over $\bQ$ if and only if $$\text{$\bfC^u = \bfC$  for all $u\in (\bZ/N_{\bfC})^*$: $\bfC$ is a {\sl rational union}.}$$ where ${}^u$  means to put each element of $\bfC$ to the power $u$. 
\end{thm} 

For this reason we have used rational unions of conjugacy classes in all examples.  Individual \MT s have an attached prime (denoted $p$ in the early papers, but we use $\ell$ here because of the latest work). 

When the \MT\ data passes a lift invariant test, then the \MT\ is an infinite (projective) system of nonempty levels (result stated precisely in our discussion of \cite{Fr06}). Each level, $\sH_k'$, has a minimal compactification $\bar \sH'_k$, that is a normal projective algebraic variety. All such compactified  levels cover the classical $j$-line, $\prP^1_j$, when $r=4$, and an $r\nm 3$ dimensional generalization, $J_r$, of it for larger $r$. 

Indeed, off the cusps, each level is a component of a {\sl reduced\/} Hurwitz space (as in \S\ref{equivalences}).  Most modern applications of algebraic equations requires more data than is given by the moduli of curves of a given genus, or even of Shimura varieties. Hurwitz spaces, however, do carry such data and retain the virtue of having moduli properties. 

\MT s come with what we call the usual \MT\ conditions: 
\begin{edesc} \label{MTconds} \item \label{MTcondsa} Each has an attached group $G$, and a collection of $r$ conjugacy classes, $\bfC$ in $G$ with $\ell'$ elements (of orders prime to $\ell$). 
\item \label{MTcondsa} Further, $G$ is $\ell$-perfect: $\ell$ divides $|G|$, but $G$ has no surjective homomorphism to $\bZ/\ell$.
\end{edesc}  For $G$ a dihedral group, with $\ell$ odd and $r=4$, we are in the case of {\sl modular curve towers}. So, \MT s generalizes modular curves towers. Since there are so many $\ell$-perfect groups, the generalization is huge. 

The Main Conjectures are these: 
\begin{edesc} \label{MainConj} \item \label{MainConja} High tower levels have general type; and 
\item  \label{MainConjb}  even if all levels have a fixed definition field $K$, finite over $\bQ$, still $K$ points disappear (off the cusps) at high levels. 
\end{edesc} Bringing particular \MT s alive plays on {\sl cusps}, as do modular curves. Cusps already appeared in Thm.~\ref{genuscomp} as the disjoint cycles of $\gamma_\infty$ (corresponding to the points over $\infty$ on the $j$-line). \S\ref{05-to-now}  of our TimeLine includes precise comparison of \MT\ cusps with those of modular curve towers, consequences of this, and two different methods that have given substantial progress on the Main Conjectures. See the argument of \sdisplay{\cite{CaTa09}}  for why \eql{MainConj}{MainConja} implies \eql{MainConj}{MainConjb} when $r=4$. 

The graphical device Ð the $\sh$(ift)-incidence matrix used in the table above Prop.~\ref{A43-2} Ð displays these cusps, and the components -- corresponding to blocks in the matrix -- in which they fall.  

Several papers emphasize, though, that cusps for Hurwitz spaces often have extra structure -- meaningful enough to suggest special names for them -- that comes from the group theory in ways that doesn't appear in the usual function theory approach to cusps. The use of the names {\sl Harbater-Mumford\/} and {\sl double identity\/} cusps in Thm.~\ref{level0MT} are examples of these.  

{\sl Shimura\/} varieties are another generalization of modular curves. They also have towers, and primes, etc. The connection of abelian varieties to \MT s has been made in several ways. There is one easily stated standout: The {\bf S}(trong) {\bf T}(orsion) {\bf C}(onjecture) on torsion points on abelian varieties implies the rational point conjecture on \MT s (see \cite{CaD08}).

It is, however, by labeling \MT\ cusps that we see tools for generalizing Serre's \OIT, especially through recognizing $\ell$-Frattini covers and using reduced Hurwitz spaces (defining tower levels). 

\subsubsection{Lessons from Dihedral groups Ð Before '95} \label{pre95} 

This section goes from well-known projects to their connection with the \MT\ program. The references to Serre's work was around two very different types of mathematics: His \OIT, with its hints of a bigger presence of Hilbert's Irreducility Theorem, and his desire to understand the difficulty of regularly realizing the \Spin\ cover of $A_n$. That so much could be made from dihedral groups, still seems amazing, though it wasn't until the formulation of \MT s, that it was possible to see that.  

This section concludes with \sdisplay{\cite{DFr94}}, which included e-mail exchanges with Mazur. It sets the stage for the division of the project into two branches during the 1st decade of the 21st Century.  

The arithmetic concentrated in the hands of Pierre D\`ebes and his collaborators Cadoret, Deschamps and Emsalem. The structure of particular \MT s -- based on homological algebra and the geometry of the spaces (cusps and components) -- follows my papers and my relation to Bailey and Kopeliovic with special impetuses from the work discussed with Liu-Osserman and Weigel. The effects of quoted work of Ihara, Matsumoto and Wewers, all present at my first talks on \MT s, is harder to classify. 

\Sdisplay{\cite{Sh64}} 
 I studied this source during my two year post-doctoral 67--69 at IAS (the {\sl Institute for Advanced Study}). Standout observation: Relating a moduli space's properties to objects represented by its points, through the {\sl Weil co-cycle condition\/}. That lead to defining the {\sl fine moduli\/} condition on absolute and inner Hurwitz spaces, and their reduced versions (respectively, \cite[Thm.~5.1]{Fr77} and \cite[\S4.3]{BFr02}).  
 
 It is from fine moduli, for example, that we draw positive solutions for a group $G$ toward the \RIGP\  from rational points on inner moduli spaces. Those respective conditions (as in \eqref{eqname}) are that for $T: G \to S_n$, the stabilizer of an integer, $G(T,1)$, in the representation $T$ is its own normalizer as a subgroup of $G$ (absolute equivalence); and $G$ is centerless (inner equivalence). The result of fine moduli is that there is a (unique) total space over the Hurwitz space representing the covers corresponding to its point.  
 
 The fine moduli condition, with the addition of reduced equivalence to Nielsen classes (Def.~\ref{redaction}) is in \cite[\S4.3]{BFr02}, as in our example Ex.~\ref{exA43-2}. 
 
Results from it: The {\sl {\bf B}(ranch){\bf C}(ycle){\bf L}(emma)\/} (see the html file in \S\ref{organization}) and its early uses starting with the solution of Davenport's problem as discussed at the beginning of \cite{Fr12}, for problems not previously considered as moduli-related. A later refined use: A model for producing \lq\lq automorphic functions\rq\rq\ on certain Hurwitz spaces as in \cite[\S6]{Fr10}, supporting the Torelli analogy through $\theta$ nulls on a Hurwitz space (\S\ref{Frat-Groth}). 

\Sdisplay{\cite{Se68}} I saw Serre give one lecture on his book while I was in my second year (1968-1969) post-doctoral at IAS. During that time his amenuenses were writing his notes. I asked them questions and I interpreted the hoped for theorem -- a little different than did Serre -- as this. For each fixed $\ell$ as $j'\in \bar\bQ$ varies, consider the field, ${}_\ell\sK_{j'}$, generated over $\bQ(j')$ by the coordinates of any projective sequence of points $${}_\ell\bx'\eqdef \{x_k'\in X_0(\ell^{k\np1})\}_{k=0}^\infty \mid \dots \mapsto x_{k\np1}'\mapsto x_k'\mapsto \dots \mapsto j'.$$ 

Though such fields don't vary smoothly, the eventual discovery was that there are two distinct types of points, called $\CM$ (for {\sl complex multiplication}) and $\GL_2$, of very different natures.  

Denote the  Galois closure of ${}_\ell\sK_{j'}/\bQ(j')$ by ${}_\ell\hat\sK_{j'}$, and its Galois group, the {\sl decomposition group\/} at $j'$, by ${}_\ell \hat G_{j'}$.\footnote{Potential confusion of notation: $j$ here is not an index, but the traditional variable used for the classical $j$-line.} Imitating the notation of the arithmetic monodromy group of a cover in Def.~\ref{arithmon}, denote the arithmetic monodromy group of the cover 
$${}_\ell \phi_{j,k}: X_0(\ell^{k\np1})\to \prP^1_j \text{ by ${}_\ell \hat G_{{}_\ell \phi_{j,k}}\eqdef {}_\ell \hat G_{j,k} $ and ${}_\ell \hat G_j$ its projective limit.} $$ 
Similarly, without the $\hat{}$, the projective limit of the geometric monodromy is ${}_\ell G_j$. For $j'\in \bar \bQ$, 
 in a natural way ${}_\ell \hat G_{j'}\le {}_\ell \hat G_j$.\footnote{The points, $\{j=0,1\}$  of ramification of the covers are special. We exclude them here, though a more precise result (due to Hilbert) includes them as \CM\ points, too.} 
In Def.~\ref{frattdef} we have what is a Frattini cover (of profinite groups). If $\psi: H\to G$ is a Frattini cover, refer to it as $\ell$-Frattini if $\ker(\psi)$ is an $\ell$-group. 

The key definition that has guided \cite{Fr18} was this. 

\begin{defn} \label{evenfratt} Refer to a sequence of covers of finite groups $$\dots \to H_{k\np1}\to H_k\to \dots \to H_1 \to H_0=G$$ as {\sl eventually Frattini\/} (resp.~eventually $\ell$-Frattini) if there is a $k_0$ for which $H_{k_0\np l} \to H_{k_0}$ is a Frattini (resp.~$\ell$-Frattini) cover for $k\ge0$. \end{defn} 

If the projective limit of the $H_k\,$s is $\tilde H$, then we just say that it is eventually Frattini since the same property will hold for any cofinal sequence of quotients. Note, too: any open subgroup of $\tilde H$ will also be eventually Frattini (resp.~$\ell$-Frattini). 

What made an impact on our approach from \cite{Se68} were these points, in the preceeding notation. 

\begin{edesc} \label{serreles} \item \label{serrelesa} For each fixed $\ell$, ${}_\ell G_j$ is eventually $\ell$-Frattini. Further, for $\ell> 3$, it is $\ell$-Frattini (right from the beginning). 
\item \label{serrelesb} If for some prime $p$, $j'\in \bar \bQ$ is not integral at $p$, then the intersection ${}_\ell \hat G_{j'}\cap {}_\ell G_j$ is open in ${}_\ell G_j$. 

\item \label{serrelesc} If a given $j'$ is of complex multiplication type (Def.~\ref{CMdef}), then the intersection of ${}_\ell \hat G_{j'}$ with ${}_\ell G_j$  is eventually $\ell$-Frattini. 

\item \label{serrelesd} From either \eql{serreles}{serrelesb} or \eql{serreles}{serrelesc}, you have only to get to a value of $k'$ with ${}_\ell \hat G_{j',k'}$ within the $\ell$-Frattini region to  assure achieving an open subgroup of the respective $\GL_2$ or $\CM$ expectation.  \end{edesc} 

The group ${}_\ell G_j$ in \eql{serreles}{serrelesa} is $\PSL_2(\bZ_\ell)$, though Serre frames his result differently, so his group is $\SL_2(\bZ_\ell)$. The distinction is not significant. 

We interpret \eql{serreles}{serrelesb} as giving an $\ell$-adic germ representating the moduli space -- through Tate's $\ell$-adically uniformized elliptic curve --- around the (long) cusp we call Harbater-Mumford on $X_0(\ell)$.  This is a model for  gleaning $G_{\bQ_\ell}$ action when $j$ is $\ell$-adically \lq\lq close to\rq\rq\ $\infty$.  

Suppose $K$ is a complex quadratic extension of $\bQ$. The technical point of complex multiplication is the discussion of 1-dimensional characters of $G_K$ on the $\bQ_\ell$ vector space -- Tate module, or 1st $\ell$-adic \'etale cohomomology -- of an elliptic curve with complex multiplication by $K$. On the 2nd $\ell$-adic \'etale cohomomology it is the cyclotomic character, while on the 1st there is no subrepresentation of any power of the cyclotomic character. 

Only a part of abelian extensions of $K$ are cyclotomic -- generated by roots of 1, a result that generalizes to higher dimensional complex multiplication in \cite{Sh64}. Much of \cite{Se68} is taken with \eql{serreles}{serrelesc}. The groups there are primarily the (abelian) ideal groups of classical complex multiplication.  

As \cite{Ri90} emphasizes, Serre's book is still relevant, especially for the role of abelian characters, those represented by actions on Tate modules (from abelian varieties), and those not. 

We reference this discussion in many places below. The full (and comfortable) completion of Serre's \OIT\ awaited replacement of an unpublished Tate piece by ingredients from Falting's Thm. \cite{Fa83} (as in \cite{Se97b}).

\Sdisplay{\cite{Fr78}} This was the forerunner of the always present relation between absolute, $\sH(G,\bfC)^\abs$, and inner, $\sH(G,\bfC)^\inn$, Hurwitz spaces \eqref{eqname}.  The latter naturally maps -- via the equivalence -- to the former. The section \cite[\S3]{Fr78} -- {\sl Determination of arithmetic monodromy from branch cycles\/} -- was based on the idea I informally call {\sl extension of constants}.   

The definition field of an absolute cover in a Hurwitz family (represented by a point $\bp\in \sH(G,\bfC)^\abs$)  would have its field extension from going to the Galois closure of the cover measured by the coordinates of a point, $\hat \bp\in \sH(G,\bfC)^\inn$  above $\bp$. \cite[Thm.~1]{FrV91}  became the standard codification of this relation. It works the same for reduced spaces. 

\cite[\S2]{Fr78} was a special case of it, where $G=D_\ell$, $\ell$ odd, and $\bfC=\bfC_{3^4}$ is 4 repetitions of the involution conjugacy class. In this case, it was describing the pair of fields $(\bQ(\bp),\bQ(\hat \bp)$), for $\bp\in \sH(G,\bfC)^{\abs}$ and $\hat \bp\in \sH(G,\bfC)^\inn$ over it. As in \S\ref{earlyOIT-MT}, this was the main case in describing prime degree ($\ell$) rational functions having the Schur cover  property \eqref{schurcover}. 

More precisely, consider the cover $f_\bp: W\to \prP^1_z$, from $\bp\in \sH(G,\bfC)^\abs$ with $\hat f_\bp$ its Galois closure. Here $W$ is isomorphic to $\prP^1_w$ over $\bQ(\bp)$ (because $\ell$ is odd). If  $\bQ(\bp)\not =\bQ(\hat \bp)$,  with extension of constants from \eql{serreles}{serrelesc} $$\text{$f_\bp$ has the Schur covering property over $\bQ(\bp)$ (Thm.~\ref{schurcoverfiber}).}$$ 

If $\bQ(\bp) =\bQ(\hat \bp)$, then $\hat f_\bp$ is an involution regular realization of $D_\ell$ over $\bQ(\bp)$ (with 4 branch points). \sdisplay{ \cite{DFr94}}, as in Thm.~\ref{Dihmain}, phrases the complete classification of involution realizations of dihedral groups. This qualifies as the \lq\lq easiest\rq\rq\  case of one of the untouched problems on the \RIGP, and straightforwardly justifies why this problem shows the \RIGP\ generalizes Mazur's Theorem. 

From each elliptic curve over $\bQ$ with non-integral $j$-invariant, the $\GL_2$ part of the \OIT\ \eql{serreles}{serrelesb}, gives explicit production of Schur cover rational functions \eqref{schurcover} of degree $\ell^2$, for infinitely many primes $\ell$. As with the \CM\ case, the distinction is measured by the difference between $\bQ(\bp)$ and $\bQ(\hat \bp)$ with $\hat\bp$ on the inner space over $\bp$ in the absolute space. 

When they are different, the degree $\ell^2$ rational function over $\bQ(\bp)$ decomposes, over $\bQ(\hat \bp)$, into two rational functions of degree $\ell$, with no such decomposition over $\bQ(\bp)$ (\cite{GMS03} and \cite[Prop.~6.6]{Fr05}). This is a phenomenon that cannot happen with {\sl polynomials\/} of degree prime to the characteristic, a fact exploited for the Schur cover property (as in  \cite{Fr70}). 

Conj.~\ref{OITgen} -- expressing our best guess for what to expect of an \OIT\ from a \MT, is the result of thinking how the relation  between these two different Schur covers compares with Serre's \OIT. Especially considering what is possible to prove at this time, both theoretically and explicitly. 

For example, the \CM\ cases are famously explicit. In particular, just as in Schur covers given by  polynomials (cyclic and Chebychev), the nature of the exceptional set $\exc_{f,K}$ in \eqref{capexc} is a union of specific arithmetic progressions (in ray class groups), and therefore it is possible to decide about compositions of exceptionals if they are exceptional.  

\begin{defn} \label{CMdef} We call $j'\in \bar\bQ$ a complex multiplication point if the elliptic curve with $j$ invariant equal to $j'$ has a rank 2  endomorphism ring. In that case that ring is identifies with a fractional ideal in a complex quadratic extension $K$ of $\bQ$.  \end{defn} 

The main point is that $G_K$ will respect those endomorphisms, and therefore it will limit the decomposition group of a projective system of points on the spaces $\{X_0(\ell^{k\np1})\}_{k=0}^\infty$. Originally, as one of Hilbert's famous problems, Kronecker and Weber used this situation to describe the abelian extensions of complex quadratic extensions of $\bQ$. 

The proof of \cite[IV-20]{Se68} concludes the proof that for $j'$ non-integral (so not complex multiplication),  the Tate curve shows there is no decomposition of the degree $\ell^2$ rational function. It even gives the following result. 

\begin{thm}[\OIT\ strong form] \label{OITsf} Suppose $j'\in \bar \bQ$ is not a complex multiplication point. Then, not only is it a $\GL_2$ point for any prime $\ell$, but the decomposition group $G_{j'}$ is actually $\GL_2(\bZ_\ell)/\{\pm1\}$ (rather than an open subgroup of this) for almost all primes $\ell$. \end{thm} 

Falting's theorem \cite{Fa83} (as in \cite{Se97b})  replaces the unpublished result of Tate. The use of Faltings in both versions of the $r=4$ Main \MT\ conjecture for \MT s mean that both have inexplicit aspects, though the results are different on that (see  \sdisplay{\cite{CaTa09}} and \sdisplay{\cite{Fr06}}). 

So even today, being explicit on Thm.~\ref{OITsf} in the Schur covering property for the $\GL_2$ case still requires non-integral j-invariant \cite[\S6.2.1]{Fr05}. \cite[IV-21-22]{Se68} references Ogg's example \cite{O67} (or \cite[\S 6.2.2]{Fr05}), to give  $j'\in \bQ$ with the decomposition group $G_{j'}$ equal $\GL_2(\bZ_\ell)/\{\pm1\}$ for {\sl all\/} primes $\ell$. 

\Sdisplay{\cite{Ih86}} The similar titles with \cite{Fr78} gives away the similar influence of Shimura. Both played on interpreting braid group actions, a monodromy action that captures data from curves, rather than from abelian varieties. 

The Ihara paper has a moduli interpretation of \lq\lq complex multiplications\rq\rq\ required to generate the field extension giving the second commutator quotient of $G_\bQ$.

Down-to-Earth result from it: Generating the second commutator (arithmetic) extensions using Jacobi sums derived from Fermat curves. Abstract result from it: An interpretation of Grothendieck-Teichm\"uller on towers of Hurwitz spaces \cite{IM95}.

\Sdisplay{\cite{Se90a}}  At the top of \S\ref{overview},  example \eql{mainexs}{mainexsb} started my interaction over this approach to the \OIT.  That expanded quickly into using the Universal Frattini cover to construct the original \MT s.  

For simplicity assume a finite group $G$ is $\ell$-perfect \eql{MTconds}{MTcondsa}. Then, the lift invariant for a prime $\ell$ described below comes from considering {\sl central\/} Frattini extensions with $\ell$ group kernels, the topic of \eql{bifurcation}{bifurcationa}. Using this gave precise statements on components of Hurwitz spaces. 

The sequence \eqref{FrVSns} is an easy to state result based on this tool,  giving a presentation of $G_\bQ$. It also produced a simply-stated conjecture. 
Assume for $K\subset \bar\bQ$ that $G_K$ is a projective profinite group.\footnote{Shafarevich's conjecture is the special case that $K$ is $\bQ$ with all roots of 1 adjoined.}

\begin{guess}[Generalization of Sharafavich's Conjecture] \label{FrVConj} Then, $K$ is Hilbertian if and only if $G_K$ is profree.\footnote{That $G_K$ profree implies it is Hilbertian is a consequence of a version of Chebotarev's field crossing argument. The \cite{FrV92} result starts with the assumption that $K$ is {\bf P}(seudo){\bf A}(lgebraically){\bf C}(losed).}  \end{guess} 

That sounds good, though it didn't lead to an understanding of how to use Nielsen classes as practical tools in many problems in algebra. For that reason we have revamped how \eqref{FrVSns} arises, recasting it as a classical mathematics connection between the \RIGP\ and the \OIT, as epitimized in \eql{bifurcation}{bifurcationb}, in \cite{Fr18}. Historical support for that is what follows here and in the next two discussions. 

Arguably, the most famous frattini central extension arises in {\sl quantum mechanics\/} from the spin cover, $\psi: \Spin_n\to O_n(\bR)$, $n\ge 3$, of the orthogonal group.  That is Wolfgang Pauli's explanation of the spin of electrons around atoms as an hermitian observable. Regard the kernal of $\psi$ as  $\{\pm 1\}$. The natural permutation embedding of $A_n$ in $O_n$  induces the $$\text{Frattini cover $\psi: \Spin_n \to A_n$, abusing notation a little.}$$ 

A braid orbit $O$ in $\ni(A_n,\bfC)$, with $\bfC$ conjugacy classes consisting of odd-order elements, passes the (spin) lift invariant test if the natural (one-one) map $\ni(\Spin_n,\bfC) \to \ni(A_n,\bfC)$ maps onto $O$. In this case, the main result of \cite{Se90a} was that if the genus attached to $\ni(A_n,\bfC)$ is 0, then the test depends only on the Nielsen class and not on $O$. The short proof of \cite[Cor.~2.3]{Fr10} is akin to the original discussion with Serre. 

Here is a particular case of this, in which we know much more.  Assume $\bfC=\bfC_{3^r}$ consists of the repetition $r$ times of the conjugacy class of 3-cycles ($r=4$ in example \eql{mainexs}{mainexsb}). Then \cite[Thm.~A]{Fr10} says that if $r=n\nm1$ (the cover has genus 0, it's minimal value by Riemann Hurwitz), then there is one braid orbit in $\ni(A_n,\bfC_{3^{n\nm1}})$. 

Not only that, but if $g\in A_n$ is a 3-cycle, then it has a unique lift $\tilde g$ as an order 3 element in $\Spin_n$, In that case $$\text{for $\bg\in \ni(A_n,\bfC_{3^{n\nm1}}), \prod_{i=1}\tilde g_i= (-1)^{n\nm1}\eqdef s_{\Spin_n/A_n}(\bg)$.}$$ 

Now consider a cover $\phi_\bg: \prP^1_w\to \prP^1_z$ representing $\bg$ as given by the conditions \eqref{bcycs}. Then, consider constructing $Z \to \hat W \to \prP^1$ with $\hat W$ the Galois closure of $\phi_\bg$, and $Z\to \prP^1_z$ Galois with group $\Spin_n$.  Result: There is  an unramified $Z\to \hat W$ if and only if $\bg$ is in the image of $\ni(\Spin_n,\bfC_{3^r})$. 

\cite[Thm.~B]{Fr10} says, for $r\ge n$, the two braid orbits on $\ni(A_n,\bfC_{3^r})$ are distinguished by their lift invariants. See \sdisplay{\cite{Fr02b}}  and  \sdisplay{\cite{We05}}. 

This example, including using the same naming of the same order lift class, $\C_3$, of elements of order 3 in both $A_n$ and $\Spin_n$, has many of the ingredients that inspired the use of the Universal Frattini cover $\tilde G$. 

The conjugacy class $\C_3$ has the same cardinality in $\Spin_n$ as it has in $A_n$. If we included, even once, a product of two disjoint 2-cycles as an element of the Nielsen class $\ni(A_n,\bfC)$, this would kill the lift invariant. These examples are one of the main considerations of the Main Theorem of \cite{FrV92}, which has been considerably revamped and expanded in \cite{Fr18}.    

Results inspired by it: There are three kinds of results affected by  the appearance of central Frattini extensions.  
\begin{edesc} \item Those that describe precisely the number of components on a Hurwitz space $\sH(G,\bfC)$ assuming the high multiplicity of appearance of each conjugacy class appearing in $\bfC$.  
\item Those that describe the precise obstruction to there being a nonempty \MT\ supported by the Nielsen class $\ni(G,\bfC)$. 
\item Those that help classify the cusps. \end{edesc}

\Sdisplay{\cite{Se90b}} A combination of this paper with \cite[\S 6]{Fr10} makes use of the lift invariant for any Nielsen class of odd-branched Riemann surface cover of the sphere in say, the Nielsen class $\ni(A_n,\bfC)$. 

It is a formula for the parity of a uniquely defined half-canonical class on any cover $\phi: W \to \prP^1_z$ in the Nielsen class that depends only on the spin lift invariant generalizing $s_{\Spin_n/A_n}$ defined above.  
From this \cite[\S6.2]{Fr10} produces {\sl Hurwitz-Torelli\/} automorphic functions on certain Hurwitz space  components through the production of even $\theta$-nulls.

\Sdisplay{\cite{Se92} and \cite{Fr94}}  
Serre didn't use the braid monodromy (rigidity) method.  Fried makes the connection to braid rigidity through Serre's own exercises. The difference shows almost immediately in considering the realizations of Chevalley groups of rank exceeding one. 

Serre records just three examples of Chevalley groups of rank exceeding one having known regular realizations at the time of his book. The technique of \cite{FrV91} and \cite{FrV92} constructed, for each finite group $G$ a covering group $G^*$, with no center, and infinitely many  collections of conjugacy classes $\bfC$ of $G^*$ with these properties: 
\begin{edesc} \label{Gconds}  \item \label{Gcondsa}  There is a (faithful) representation $T^*: G^*\to S_{n^*}$ for which the stabilizer of 1, $G^*(T^*,1)$ is self normalizing.  \item \label{Gcondsb}  $N_{S_{n^*}}(G^*)/G^*$ consists of all outer automorphisms of $G^*$. 
\item \label{Gcondsc}  The corresponding  inner Hurwitz spaces $\sH(G^*,\bfC)^\inn$ are irreducible and have definition field $\bQ$.   
\end{edesc}  

This allowed using the Hurwitz spaces as part of a {\sl field-crossing\/} argument over any {\bf P}(seudo){\bf A}(lgebraically){\bf C}(losed) field $F\subset \bar \bQ$ -- any absolutely irreducible variety over $F$ has a Zariski dense set of $F$ points. The result was that if $F$ was also Hilbertian, then $G_F$ is profree (see Conj.~\ref{FrVConj}), and a particular corollary was the presentation of $G_\bQ$ in \eqref{FrVSns}. 

Condition \eql{Gconds}{Gcondsa} is sufficient to say that any $K\subset \bar \bQ$ point on $\sH(G^*,\bfC)^\inn$ (satisfying \eql{Gconds}{Gcondsc}) corresponds to a $K$ regular realization of $G^*$, and therefore of $G$. This is because  $G^*$ will have no center, the condition that the inner Hurwitz space is then a fine moduli space. 

In myriad ways we can relax these conditions. Still, to use them effectively over say $\bQ$ requires finding $\bQ$ points on $\sH(G^*,\bfC)^\inn$. The usual method is to choose $\bfC$ so that $\sH(G^*,\bfC)^\inn$ is sufficiently close to the configuration space $U_r$, that $\bQ$ points are dense in it. If $r=4$, we may use Thm.~\ref{genuscomp} to compute the genus of $\sH(G^*,\bfC)^{\inn,\rd}$ and check the possibility it has genus 0, with good reason for it to have (at least one, so $\infty$-ly many) $\bQ$ points. 

Soon after \cite{FrV91}, V\"olklein and Thompson -- albeit powerful group theorists -- produced high rank Chevalley groups in abundance based on this method. Locating specific high-dimensional uni-rational Hurwitz spaces was the key here. Examples, and the elementary uses of Riemann's Existence Theorem, abound in \cite{Vo96}.

The Conway-Fried-Parker-Voelklein appendix of the Mathematische Annalen paper was a non-explicit method for doing that. \cite{Fr10}  shows what it can mean to be very explicit about this. 

\Sdisplay{\cite{DFr94}}  Thm.~\ref{Dihmain} gave the formulation of the Main \MT\ conjecture for dihedral groups. Equivalently, the \BCL\ (Thm.~\ref{bclthm}) implies there is one even integer $r (\le r^*)$ and for each $k\ge 0$,  a dimension ${r\nm 2}\over 2$ {\sl hyperelliptic Jacobian\/} (over $\bQ$) with a $\bQ(e^{2\pi i/\ell^{k\np1}})$ torsion point, of order $\ell^{k\np1}$,  on whose group $G_\bQ$ acts as it does on $\lrang{e^{2\pi i/\ell^{k\np1}}}$. 

The {\sl Involution Realization Conjecture\/} says the last is impossible: There is a uniform bound as $n$ varies on $n$  torsion points on any hyperelliptic Jacobian of a fixed dimension, over any given number field. (The only proven case, $r=4$, is the Mazur-Merel result bounding torsion on elliptic curves.) If a subrepresentation of the cyclotomic character occurred on the $\ell$-Tate module of a hyperelliptic Jacobian (see \cite{Se68}), the Involution Realization Conjecture would be blatantly false. 

This soon led to the formulation of the Main \MT\ conjecture in the discussion of \sdisplay{\cite{FrK97}}. Still missing: For even a single prime $\ell > 2$, find such cyclotomic $\ell^{k\np1}$ torsion points on any hyperelliptic Jacobian for all (even infinitely many) values of $k$.

\subsubsection{Constructions and Main Conjectures from 1995 to 2004} \label{95-04} Recall the definitions of Frattini cover $H\to G$ of groups (Def.~\ref{frattdef}) and $\ell$-perfect for a group $G$. In \S\ref{modtowdef} we alluded to a universal $\ell$-Frattini cover ${}_\ell\psi: {}_\ell\tilde G\to G$ for any group $G$. Here we construct it. 

Among its properties it is the minimal profinite cover of $G$ for which its $\ell$-Sylow is a pro-free pro-$\ell$ group. If $P$ is an $\ell$-group, then its Frattini subgroup is ${}_{\text{\rm fr}}P\eqdef P^\ell[P,P]$. As usual, we understand ${}_{\text{\rm fr}}P$ to be the closed subgroup of  $P$ with generators from the $\ell$ powers and commutators of $P$. Then,  $P\to P/{}_{\text{\rm fr}}P$ is a Frattini cover. 

Recover a cofinal family of finite quotients of $\fG \ell$ by taking $\ker_0=\ker({}_\ell \psi)$ denoting the sequence of {\sl characteristic kernels\/} of $\fG \ell$ as in \eqref{charlquots}: \begin{equation} \label{charlquots2} \ker_0> {}_{\text{\rm fr}} \ker_0\eqdef  \ker_1 \ge \dots \ge  {}_{\text{\rm fr}}\ker_{k{-}1} \eqdef \ker_k \dots,\end{equation}  $\fG \ell/\ker_k$ by $\tfG \ell k$, and the characteristic modules $\ker_k/\ker_{k'}={}_\ell M_{k,k'}$, etc.  

\Sdisplay{\cite{Fr95}}   Assume generating conjugacy classes, $\bfC$, of $G$.\footnote{The group generated by all entries of $\bfC$ is $G$.}  Then, with $N_{\bfC}$ the least common multiple of the orders of elements in $\bfC$:  
\begin{equation} \label{congcond} \begin{array}{c} \text{If $\ell \not| N_{\bfC}$,  Schur-Zassenhaus implies the classes $\bfC$ lift canonically}\\ \text{ to classes of elements of the same orders in each group $\tfG \ell k {}$.}  \end{array}\end{equation} 

This paper opens by recasting modular curves as Hurwitz spaces of sphere covers for the dihedral group by referring to their use in \sdisplay{\cite[\S2]{Fr78}}. Then, upon applying the construction of \eqref{charlquots2}, that any group can be used to constructed modular curve-like towers. To make the case that this is worth doing it considers these topics. 

\begin{edesc} \label{reason} \item \label{reasona}  That this works significantly with essentially any group $G$ and prime $\ell||G|$, replacing a dihedral group $D_\ell$, $\ell$ odd, and generating conjugacy classes $\bfC$ satisfying \eqref{congcond}.  
\item \label{reasonb}  That the resulting tower of spaces can be expected to have significant properties comparable to modular curve towers, useful for investigating typical problems considered in arithmetic geometry. 
\item \label{reasonc}  That following this path recasts the entire role of the \RIGP, and possibily eventually the \OIT, on substantive moduli spaces, that have arisen in the works of others.     
\end{edesc} 

Without \eqref{congcond}, there is no unique assignment of lifts of classes in $\bfC$ to the characteristic $\ell$-Frattini cover groups. That stems from this. If $g\in G$ has order divisible by $\ell$, then the order of any lift $\tilde g\in \tfG \ell 1$ is $\ell\cdot \ord(g)$.

Given \eqref{congcond}, we may canonicially form towers of Nielsen classes, and their associated Hurwitz spaces, from the sequences of groups in \eqref{charlquots} and their abelianizations: 
\begin{equation} \label{MTHur}  \text{$\{\sH(\tfG \ell k {},\bfC)^\inn\}_{k=0}^\infty$ and the abelianized version $\{\sH(\tfG \ell k {}\ab,\bfC)^\inn\}_{k=0}^\infty$.}\end{equation}  Originally we called these the \MT s. Now we prefer that a \MT\ is a projective sequence of irreducible components (from braid orbits on the Nielsen classes) of their respective levels.  

To address \eql{reason}{reasona}, \cite[Part II]{Fr95} describes the characteristic modules for $G=A_5$ and each prime $\ell=2,3,5$ dividing $|A_5|=60$. Thereby, for these cases, it describes the tower of Nielsen classes attached to the abelianized series for \eqref{charlquots} obtained from the same series using the $\ell$-Frattini cover 
$${}_\ell\psi_\ab : {}_\ell\tilde G_\ab\eqdef {}_\ell\tilde G/[{}_\ell\psi,{}_\ell\psi]\to G,$$ with similar notation for the series $\tfG \ell k {}_\ab$, $k\ge 0$.

We then required three immediate assurances. 
\begin{edesc} \label{mtneed} \item   \label{mtneeda} That we could decide when we are speaking of a non-empty \MT. 
\item  \label{mtneedb} That $K$ points on the $k$th tower level correspond to $K$ regular realizations in the Nielsen class  $\ni((\tfG \ell k {},\bfC)$ (or  $\ni((\tfG \ell k {}_\ab,\bfC)$).  
\item \label{mtneedc} That we know the definition field of $\sH(\tfG \ell k {},\bfC)^\inn \to U_r$ and the rest of the structure around $\sH(\tfG \ell k {},\bfC)^\inn$ as a moduli space. \end{edesc} 

\begin{proof}[Comments]  Response to \eql{mtneed}{mtneeda}:  The first necessary condition is that $G$ is $\ell$-perfect. Otherwise, no elements of $\bfC$ will generate a $\bZ/\ell$ image. 

A much tougher consideration, though, was what might prevent finding elements $\bg\in G^r\cap \bfC$ satisfying product-one (as in \S\ref{analgeom}). \sdisplay{\cite{Fr02b}}  and  \sdisplay{\cite{We05}}  discuss the final resolution of that using the {\sl lift invariant}. 

Response to \eql{mtneed}{mtneedb}:  Originally I formed \MT s to show that talking about rational points on them, vastly generalized talking about rational points on modular curve towers. Especially, that the \RIGP\ was a much tougher/significant problem than usually accepted. 

To assure $K$ points on the $k$th level correspond to regular realizations of the Frattini cover groups, we needed the fine moduli condition that each of the $\tfG \ell k\,$s has no center. The most concise is this \cite[Prop.~3.21]{BFr02}: 
\begin{equation} \label{modcond} \text{If $G$ is centerless, and $\ell$-perfect, then so is each of the $\tfG \ell k\,$s.} \end{equation} 

Response to \eql{mtneed}{mtneedc}: The \BCL\ of \cite[\S5.1]{Fr77} perfectly describes the definition fields of the Hurwitz spaces. The result is more complicated for absolute classes, but both results also apply to reduced classes.  We have given detailed modern treatments of this now in several places including \cite[App.~B.2]{Fr12}; also see \cite[Main Thm.]{FrV91}~and \cite{Vo96}. 

Denote the least common multiple of $\{\ord(g)\mid g\in \bfC\}$ by $N_{\bfC}$, and the field generated over $K$ by a primitive $N_{\bfC}$ root of 1 by $\Cyc_{K,\bfC}$ . Recall: $$G(\Cyc_{\bQ,\bfC}/\bQ)=(\bZ/N_{\bfC})^*, \text{ invertible integers } \!\!\!\mod N_{\bfC}.$$ Then,  $G(\Cyc_{K,\bfC}/K)$ is the subgroup fixed on $K\cap \Cyc_{\bQ,\bfC}$. We define two cyclotomic fields. 
\begin{equation} \label{bcl} \begin{array}{rl} &\bQ_{G,\bfC} \!\eqdef\{m\in (\bZ/N_{\bfC})^* \!\!\mid \{g^m\mid g \in \bfC\} \eqdef\, \bfC^m\, =\,\bfC\}. \\ 
&\bQ_{G,\bfC,T} \!\eqdef \{m\in (\bZ/N_{\bfC})^* \!\!\mid \exists h\in N_{S_n}(G,\bfC) \text{ with } h \bfC^m h^{-1} =\bfC\}. \end{array} \end{equation} 

\begin{lem}[Branch Cycle] \label{branchCycle} As above, then $\bQ_{G,\bfC}$ (resp.~$\bQ_{G,\bfC,T}$) is contained in any definition field of any cover in the Nielsen class $\ni(G,\bfC)^\inn$ (resp.~$\ni(G,\bfC,T)^\abs\eqdef \ni(G,\bfC)^\abs$ if $T$ is understood) \cite[p.~62--64]{Fr77}.\end{lem} 

Still, for $K$ points to exist, there must be a $K$-component (as a moduli space). That is a much harder problem, for there are good reasons there can be more than one component (as in the discussion of \sdisplay{\cite{Se90b}} on \cite[Thm.~B]{Fr10})  in any particular case. The \OIT\ contends with that at all levels, as the discussion of $\star$ \cite{FrH15} $\star$ shows. 

In lieu of the Main \MT\ conjecture \ref{MTconj}, for $K$ a definition field of the $k=0$ level of a given \MT, $\{\sH_k\}_{k=0}^\infty$, there are really two different types of cases.   Towers where there is no K for which all levels have definition K, and towers for which the definition fields of the levels have no bounded degree.  Again, we contend with both types in considering the \OIT. \end{proof} 

\begin{guess}[Main \MT\ Conjecture] \label{MTconj} For $K$ a number field, at high levels there will be no $K$ points on a \MT. Also, high levels will be algebraic varieties of general type (high powers of the canonical bundle embed the variety in projective space) \cite{Fr95}.  \end{guess} 

To address \eql{reason}{reasonb} \sdisplay{\cite{BFr02}} went after the Main Conjecture \eqref{MTconj} by inspecting the properties of the $(A_5,\bfC_{3^4},\ell=2)$ case in sufficient detail that any fair observer could see there was something substantive happening in essentially any \MT. 

\cite[Thm.~3.21]{Fr95} got the most attention, for which  we need the following. 
\begin{defn} \label{HMrep} An element $\bg\in \ni(G,\bfC)$ is a {\sl Harbater-Mumford\/} (\HM) representative if it has the form $(g_1,g_1^{-1},\dots,g_s,g_s^{-1})$ (so $2s=r$). A braid orbit $O$ is said to be \HM, if the orbit contains an \HM\ rep. \end{defn} 

The result showed that if $\bfC$ is a rational union, then $G_\bQ$ permutes the \HM\ components. Further, it gave an explicit criterion for showing  there was just one $\HM$ component, that applied to any group $G$. Thereby, it found for $G$, reasonably small values of $r$ for which there was a \MT\ attached to $(G,\bfC)$, for some $\bfC$, with $\bQ$ as a definition bound on all the \MT\ levels. Below we see many picked up on this for compactifying Hurwitz spaces. 

\Sdisplay{\cite{FrK97}}  Suppose $G$ is a group with many known regular realizations. For example:  $A_n$  semidirect product some finite abelian group (like a quotient of $\bZ^{n-1}$ on which $A_n$ acts through its standard representation; see http://www.math.uci.edu/~mfried/deflist-cov/RIGP.html, \S IV.1). Consider, for some prime $\ell$ for which $G$ is $\ell$-perfect, if there are regular realizations of the whole series of $\tfG \ell k$, $k\ge 0$, over some number field $K$.  The basic question: Could all such realizations have a uniform bound, say $r^*$, on the number of branch points -- with no hypothesis on the classes $\bfC$. 

\begin{thm} Such regular realizations are only possible by restricting to $\ell'$ classes (elements of $\bfC$ have orders prime to $\ell$). If they do occur, there must exist a \MT\ over $K$ for some one choice of $r\le r^*$ classes, $\bfC$, with a $K$ point at every level \cite[Thm.~4.4]{FrK97}. \end{thm} 

This Fontaine-Mazur analog \cite{Fr06b}, generalizes for each such $G$ the {\sl Involution Realization Conjecture\/} for dihedral groups (as in 
\sdisplay{ \cite{DFr94}}). Of course, the conclusion in the result is contrary to the Main \MT\ Conjecture \ref{MTconj}, which has been proven for the case $r^*\le 4$. 
\vskip.2in

\Sdisplay{\cite{BFr02}} This is a book of tools that has informed all later papers on \MT s. The thread through the book is checking phenomena on \MT s lying over one (connected) reduced Hurwitz space: For the Nielsen class $\ni(A_5,\bfC_{3^4})$; four repetitions of the conjugacy class of 3-cycles) and the prime $\ell=2$. It computes everything of possible comparison with modular curves about level one (and level 0). 

It shows the Main \MT\  Conjecture \ref{MTconj} for it: No $K$ points at high levels (K any number field). The inner space at level 0 has one component of genus 0. Level one has two components, of genus 12 and genus 9. This concludes with a conceptual accounting of all cusps, and all real points on any \MT\ over the level 0 space (none over the genus 9 component). A version of the spin cover (extending the domain of use of \cite{Se90a}) obstructs anything beyond level 1 for the genus 9 component. 

Much is made of this argument: Any prime $\ell$ of good reduction, for which there are $\bZ/\ell$ points at each level of a \MT, would automatically give the trivial power of the cyclotomic character acting on a Tate module, as disallowed in \cite{Se68}.

\cite[\S2.10.2]{BFr02} introduces a very handy device for detecting braid orbits, and organizing the nature of cusps. We call it the {\sl shift-incidence matrix\/}.  

Recall the braid generators in \eqref{Hrgens}. Choose any one of the twists $q_v$  (for $r$ = 4 it suffices to choose $q_2$ on reduced Nielsen classes) and call it $\gamma_\infty$. The rows and columns of the matrix are referenced by the orbits of $\gamma_\infty$ on reduced Nielsen classes as $\row O t$. Reduced Nielsen classes are special in the case $r=4$, as in  Thm.~\ref{genuscomp}.   

The $(i,j)$ entry of the matrix is then $|(O_i)\sh\cap O_j|$ indicating that you apply $\sh$ to all entries of $O_i$, intersect it with $O_j$, and put the cardinality of the result in the $(i,j)$-entry. 

Since $\sh^2=1$ on reduced classes when $r=4$, for that case the matrix is symmetric. Braid orbits correspond to matrix blocks. \cite[Table 2]{BFr02} displays the one block and the genus calculation for $(A_5 , \bfC_{3^4})$. Then, \cite[\S8.5, esp.~Table 4]{BFr02} does a similar calculation for the level 1 \MT s, $({}_2^1 A_5, \bfC_{3^4})$.  

In this example we see the refined analysis that allows us to understand \MT\ levels through their cusps. 

\begin{edesc} \item There are two kinds of cusps, \HM\ and near-\HM, an examplar of the cusp types that have occurred in all examples, with near-\HM\ having a special action under the complex conjugation operator. \item With $r=4$,  from the genus of the components at level 1 being higher than 1, at high levels there can by no $K$ points. \end{edesc} 

\cite[Ex. A.3]{Fr10} shows  $\sH(A_4,\bfC_{+3^2-3^2})^{\inn,\rd}$ has two components (both genus 0) using the shift incidence matrix (computed in \cite[Prop.~3.5]{Fr10}). We have reported on that in Ex.~\ref{exA43-2}, where you can see directly from the cusps that appear in the $\sh$-incidence matrix that the components have respective degrees 9 and 6 over the $j$-line.

\Sdisplay{\cite{Fr02b}   and   \cite{We05}}   
\cite[Chap.~9]{Se92} added material from \cite{Me90} on regularly realizing the $\psi_{A_n,\Spin_n}: \Spin_n\to A_n$ cover. Stated in my language he was looking at the Nielsen class extension $$\Phi_{A_n,\Spin_n}: \ni(\Spin_n, \bfC_{3^{n-1}})^\inn \to \ni(A_n, \bfC_{3^{n-1}})^\inn.$$

\begin{thm} \label{Anlift} \cite[Main Thm.]{Fr10}  For all $n$, there is one braid orbit for $\ni(A_n, \bfC_{3^{n-1}})^\inn$. For $n$ odd, $\Phi_{A_n,\Spin_n} $ is one-one, and the abelianized \MT\ is nonempty. For $n$ even, $\ni(\Spin_n, \bfC_{3^{n-1}})^\inn$ is empty. \end{thm} 

The meaning of the $n$ even case is this. For $\bg \in  \ni(A_n, \bfC_{3^{n-1}})^\inn$, if you lift its entries to same order entries in $\Spin_n$, to get $\hat \bg$, the result does not satisfy product-one: $\hat g_1, \dots, \hat g_{n-1} = -1$: the {\sl lift invariant\/} in this case. 

I used this to test many properties of \MT s. Here it showed that there is a {\sl  nonempty\/} Modular Tower over $\ni(A_n, \bfC_{3^{n-1}})^\inn$ for $\ell=2$ if and only if $n$ is odd.  In particular the characteristic Frattini extensions define the tower levels, but central Frattini extensions control many of their delicate properties. If you change $\bfC_{3^{n-1}}$ to $\bfC_{3^{r}}$, $r\ge n$, there are precisely two components, with one obstructed by the lift invariant, the other not. 

\cite[Thm.~2.8]{Fr02b} gives a procedure to describe the $\ell$-Frattini module for any $\ell$-perfect $G$, and therefore of the sequence $\{\tfG \ell k {}_\ab\}_{k=0}^\infty$. \cite[Thm.~2.8]{Fr02b}  labels Schur multiplier types, especially those called {\sl antecedent}. Example: In \MT s where $G=A_n$, the antecedent to the level 0 spin cover affects \MT\ components and cusps at all levels $\ge  1$ (as in \cite{Fr06}). 

\cite[I.4.5]{Se97a} extends the classical notion of {\sl Poincar\'e duality\/} to any pro-$\ell$ group. Especially it was applied to the pro-$\ell$ completion of the fundamental group of a compact Riemann surface of any given genus. \cite{We05} uses the extended notion, intended for groups that have pro-$\ell$ groups as extensions of finite groups. Main Result: The universal ${\ell}$-Frattini cover $\fG \ell$ (and $\fG \ell {}_\ab$) is an $\ell$-Poincar\'e duality group of dimension 2.  Now compare Thm.~\ref{genspinn} with the $\Spin_n\to A_n$ case. 

\begin{thm} \label{genspinn}  \cite[Cor.~4.19]{Fr06}: Assume the usual \MT\ conditions, $G$ that is $\ell$-perfect, $\ell'$ classes $\bfC$, and let $O$ be a braid orbit on the Nielsen class. Take $R_{\ell,G}\to G$ to  be the maximal $\ell$-central Frattini extension of $G$ (an $\ell$-representation cover). Then, there is a (nonempty) abelianized \MT \ over the Hurwitz space component corresponding to $O$ if and only if the natural (one-one) map $\ni(R_{\ell,G},\bfC) \to  \ni(G,\bfC)$ is onto $O$.  \end{thm} 

\subsubsection{Progress on the \MT\ conjectures and the \OIT} \label{05-to-now}  
As with modular curves, the actual \MT\ levels come alive by recognizing moduli properties attached to particular (sequences) of cusps. It often happens with \MT s that level 0 of the tower has no resemblance to modular curves, though a modular curve resemblance arises at higher levels.

Level 0 of alternating group towers illustrate: They have little resemblance to modular curves. Yet, often level 1 starts a subtree of cusps that contains the cusptree of modular curves. We can see this from a (preliminary) classification of cusps discussed in \sdisplay{ \cite[\S3]{Fr06}}. 

Here, near the end of this paper,  my intention is to leave the deeper part of the discussion of generalizing Serre's \OIT\ to \cite{Fr18}. Still, we make one point here, based on our long discussion  in \sdisplay{\cite{Se68}} and \sdisplay{\cite{Fr78} }.  

With $\dagger$ either inner or absolute equivalence,  it is the interplay of two Nielsen classes that gives a clear picture of the bifurcation between the two types of decomposition groups, \CM\ and $\GL_2$. Those Nielsen classes are $$\text{$\ni(D_\ell,\bfC_{2^4})^{\dagger,\rd}$ and $\ni((\bZ/\ell)^2\xs \bZ/2,\bfC_{2^4})^{\dagger,\rd}$.}$$  

In the example(s) of \cite{Fr18}, the same thing happens there, too. Of course, the $j$ values giving the differentiation in types won't be precisely the same as that for Serre's modular curve case. Further, as in \sdisplay{\cite{FrH15}}, there are nontrivial lift invariants, and more complicated, yet tractible, braid orbits. 

\Sdisplay{\cite{D06}, \cite{W98},  \cite{DDes04}  and \cite{DEm05}}  \cite{D06}  has expositions on \cite{DFr94}, \cite{FrK97}, \cite{DDes04} and \cite{DEm05}, in one place. 

See the definition of Harbater-Mumford component in the comments on \eqref{mtneed}. It was by using a compatification of the Hurwitz space that one could see if there was only one Harbater-Mumford component, then $G_\bQ$ would fix that component, thus showing it is defined over $\bQ$. 

A driver of all this is \cite[Thm.~3.21]{Fr95} which constructs \MT s all of whose levels are defined over $\bQ$, each a Harbater-Mumford  component of fine moduli inner Hurwitz spaces. In particular, whose $k$-th level rational points correspond to regular realization of the $k$-th characteristic quotient $\tfG \ell k$ of the universal $\ell$-Frattini cover of $G$.

I  point out from whence comes the name {\sl Harbater-Mumford}. \cite{Mu72} used a completely degenerating curve on the boundary of a space of curves. In a sense, \cite{Ha84}, for covers, consists of a \lq\lq germ\rq\rq\ of such a construction.  

\cite{W98} developed a Deligne-Mumford stable-compactification of Hurwitz spaces that would put a Harbater degeneration on the boundary of the space. This would allow a standard comparison -- contrasting with the group theoretic use of {\sl specialization sequences\/} in Fried's result -- for labeling Harbater-Mumford cusps as lying on Harbater-Mumford components. Both compactifications are compatible with the \MT\ construction (they form natural projective systems). Further, they support -- from Grothendieck's famous theorem -- that, other than primes $p$ dividing $|G|$,  the whole \MT\ system has good reduction $\mod p$. 

\cite{DDes04}   considers that if  the arithmetic Main \MT\ Conjecture \eqref{MainConj} were wrong, then there would be a finite group $G$  satisfying the usual conditions for $\ell$ and $\bfC$ so that for some number field $K$, the corresponding \MT\ would have a $K$ point at every level. Using the \cite{W98} compactifications of the \MT\ levels, for almost all primes $\bp$ of $K$, this would give a projective system of $\sO_{K,\bp}$ (integers of $K$ completed at $\bp$) points on cusps. The results here considered what \MT s (and some generalizations) would support such points for almost all $\bp$ using Harbater patching (from \cite{Ha84}) around the Harbater-Mumford cusps. 

\cite{DEm05}  continues the results of \cite{DDes04}, ties together the notions of Harbater-Mumford components and the points on cusps that correspond to them, connecting several threads in the theory. As an application, they construct, for every projective system $\{G_k\}_{k=0}^\infty$, a tower of corresponding Hurwitz spaces, geometrically irreducible and defined over $\bQ$ (using the criterion of \cite[Thm.~3.21]{Fr95}). These admit  projective systems of points over the Witt vectors with algebraically closed residue field of $\bZ_p$, avoiding only those $p$ dividing some $|G_k|$. 

\cite[Fratt.~Princ.~2]{Fr06} says existence of a g-$\ell'$ cusp defines a regular realization of $\fG \ell$  over any algebraic closure of $\bar \bQ$ in the Nielsen class. Likely this is {\sl if and only if}. The approach to more precise results has been to consider a Harbater patching converse: Identify the type of a g-$\ell'$ cusp that supports a Witt-vector realization of $\fG \ell$. 

\Sdisplay{\cite{Fr06}} Also see previous discussions on this paper, as in \sdisplay{\cite{Fr02b},  \cite{We05}  and  \cite{DEm05}}.   Cusp types and Cusp tree on a Modular Tower: If you compactify the tower levels, you get complete spaces, with cusps lying on their boundary. The \MT\ approach allows identifying these cusps using elementary finite group theory. They are three main types $$\text{ \cite[\S 3.2.1]{Fr06}: $\ell$-cusps, g(roup)-$\ell'$ and o(nly)-$\ell'$.}$$ Modular curve towers have only the first two types, with the g-$\ell'$  cusps on them the special kind called shifts of H(arbater)-M(umford). Let $O$ be a braid orbit on $\ni(G,\bfC)$. 

\begin{thm} There is a full \MT\ over the Hurwitz space component corresponding to $O$ if $O$ contains a g-$\ell'$ representative (no need to check central Frattini extensions as in Thm.~\ref{genspinn}). \end{thm} 

For $r=4$, there is this type of modular curve result. If $O$ contains an  \HM\  cusp that is also an $\ell$-cusp, the Main Conjecture \eql{MainConj}{MainConja} holds explicitly for any \MT\ over $O$ in this sense. The genus of the tower levels grows rapidly. 
Generalizing the precise \MT\ criterion given in \cite[Princ.~4.23]{We05} gives a lift invariant criterion for $\ell$-cusps to lie over o-$\ell'$  cusps.

\Sdisplay{\cite{CaTa09}} We discuss the two approaches to the Main Conj.~\ref{mainconj} that give its truth for $r=4$, Thm.~\ref{truer=4}. Both use Falting's Theorem. First, \cite{CaTa09}, which is more general and far less explicit. As those authors admit, this was motivated by the Main Conjecture \ref{MTconj}, which Tamagawa saw at my lectures in the late '90s at RIMS.  

Let $\chi: G_K \to  \bZ_p^*$ be a character, and $A[p^\infty](\chi)$ be the $p$-torsion on an abelian variety $A$ on which the action is through $\chi$-multiplication. Assume $\chi$ does not appear as a subrepresentation on any Tate module of any abelian variety (see \cite{Se68}, \cite{DFr94} and \cite{BFr02}). Then, for $A$ varying in a 1-dimensional family over a curve $S$ defined over $K$, there is a uniform bound on $|A_s[p^\infty](\chi)|$ for $s \in S(K)$. In particular, this gives the Main \MT\ conjecture when $r=4$. 

By contrast, the \MT\ approach has specific objects (tower levels) for which the disappearance of rational points in support of the Main Conjecture, engages this explicit topic: 
If an \RIGP\ realization of $G$ exists, where is it? From \cite{FrV92} (and related), such must correspond to $\bQ$ points on Hurwitz spaces, and we are certainly coming to understand this topic. Despite compatibilities with the Cadoret-Tamagawa appoach, the biggest difference is that \cite[Prop. 5.15]{Fr06} displays a \MT\ level where the genus exceeds 1. The effectiveness of this result uses the classification of the cusps (as in  \sdisplay{\cite{Fr06}}) and the explicit genus computation of Thm.~\ref{genuscomp}. 

\cite[\S5.5]{BFr02} is a model for this as applied, say, to the sequence \eqref{A52}. So, \cite{Fa83} implies this level has, for any $K$, but finitely many rational points. 

If there were points at all higher levels, from the Tychonoff Theorem, some subset of them would be part of a projective system of $K$ points on the \MT. By applying Grothendieck's good reduction result, we could reduce the tower modulo a prime (distinct from $\ell$) and get a projective system of points on it over a finite field. This would contradict Weil's Theorem for the action of the Frobenius on  the Jacobian of the curve corresponding to that projective system. Unlike the growing genus result, this is not explicit for deciding at what level those finitely many $K$ points would disappear.  

\Sdisplay{\cite{CaD08}}  A serious topic has been missing from \RIGP\ vs \IGP\ discussions. A simple statement would to ask why  the \RIGP\ (combined with Hilbert's Irreducibility Theorem) has been so much more successful in realizing groups as Galois groups, and in connecting to other research areas? 

That does, however, leave out that -- in the form of \MT s -- both areas have a central place for nilpotent groups. A more complicated statement would consider what advantage there is to \RIGP\ realizations, and how one would decide among realizations of a given $G$ over $\bQ$, whether there was evidence for an \RIGP\ realization. 

This topic that has been gaining recent attention. It will make an appearance in \cite{Fr18}.   The \cite{CaD08} contribution is this. Suppose a finite group $G$ has a regular realization over $\bQ$. Then the abelianization of its $p$-Sylow subgroups has order ($p^u$) bounded by an expression in their index $m$ in $G$, the branch point number $r$ and the smallest prime $v$ of good reduction of the cover. This is a new constraint for the \RIGP. 

To whit: If $p^u$ is large compared to $r$ and $m$, the  branch points of the cover must coalesce modulo some prime $v$; a $v$-adic measure of proximity to a cusp on the corresponding Hurwitz space. Here is a striking conjecture. 
\begin{guess} Some expression in $r$ and $m$, independent of $v$, bounds $p^u$. This follows from the S(trong) T(orsion) C(onjecture) on abelian varieties, and it gives forms of the Main \MT\ conjecture. \end{guess} 

\Sdisplay{\cite{LO08}  and  \cite{Fr09}}  This gave an especially good place to see the sh-incidence matrix (discussion of \sdisplay{\cite{BFr02}}) in action on a variety of cusps with extra structure inherited from conjugacy classes in $A_n$. 

\cite{LO08} Showed the absolute Hurwitz spaces of pure-cycle (elements in the conjugacy class have only one length $\ge$ 2 disjoint cycle) genus 0 covers have one connected component. There is a conspicuous overlap with the 3-cycle result of \cite[Thm.~1.3]{Fr10}, the case of four 3-cycles in $A_5$.  \cite[\S5]{LO08} gives the impression that all these Hurwitz spaces are similar, without significant distinguishing properties. \cite[\S9]{Fr09}, however, dispels that. 

First it notes that subsets of these Nielsen classes can have differing inner Hurwitz spaces, varying in having one or two components. Then, by detecting seriously diverging behaviors in their level 1 cusps. 

The stronger results come by considering the inner (rather than absolute) Hurwitz spaces. \cite[Prop.~5.15]{Fr09} uses the sh-incidence matrix to display cusps, elliptic fixed points, and  genuses of the inner Hurwitz spaces in two infinite lists of \cite{LO08} examples. In one there are two level 0 components (conjugate over a quadratic extension of $\bQ$). For the other just one. 

Further, applying \cite[\S 3]{Fr06}, the nature of the 2-cusps in the \MT s over them differ greatly. None have 2-cusps at level 0. For the list with level 0 connected, the tree of cusps, starting at level 1, contains a subtree isomorphic to the cusp tree on a modular curve tower: it has a {\sl spire}. For the other list, there are 2-cusps, though not quite like those of modular curves. 

A spire makes plausible a version of Serre's \OIT\ on such MTs. The goal of recognizing $\ell$-cusps is a key to showing high \MT\ levels have general type, and it gives an approach to the Main Conjectures for $r\ge 5$. 

\Sdisplay{\cite{FrH15}} The culminating topic of \cite{Fr18}  is the system of \MT s based on the Nielsen class $\ni_{\ell,3}\eqdef \ni((\bZ/\ell)^2\xs \bZ/3, \bfC_{+3^2-3^2})$.  Notice our choice of conjugacy classes (at first in $\bZ/3$, but extended to $\ni_{\ell,3}$) is a rational union, as given by Thm.~\ref{bclthm}. From this we have Hurwitz spaces (Nielsen classes) with their extra structure defined over $\bQ$ using the \BCL. 

If we follow the general statement for a \MT\ that a prime $\ell$ is considered only if $G$ is $\ell$-perfect, then our condition would be $(\ell,3)=1$, and for each $\ni_{\ell,3}$, only that prime would be involved in the \MT. Since, however, we wish to follow the style of Serre's \OIT, we must make an adjustment. 

\begin{edesc} \label{3vs2}  \item \label{3vs2a} As Serre includes all primes, including $\ell=2$, wouldn't it be best if we included all primes, too, including $\ell=3$?  
\item \label{3vs2b} Then, however, for each $\ell$, why would we leave out the prime 3 in applying to the particular case of $\ni_{\ell,3}$? \end{edesc} 

The trick is this: Since in Serre's case (resp.~our case), the $\bZ/2$ (resp.~$\bZ/3$) is a splitting coming from a semi-direct product $\bZ\xs \bZ/2$ (resp.~$(\bZ/2)^2\xs \bZ/3$, we are free to ignore the copy of $\bZ/2$ (resp.~$\bZ/3$) even if $\ell=2$ (resp.~$3$). For the same reason we ignore that prime if $\ell$ is not 2 (resp.~3). 

Below we quote only the level 0 braid orbit description, but for all $\ell$. There are several different \MT s in our case. This doesn't occur in dihedral group cases, because there is no central $\ell$-Frattini cover of $D_\ell$ (for $\ell$ odd). So, no lift invariant occurs in the Nielsen class that starts Serre's \OIT. There is, though as in Thm.~\ref{Anlift} in the alternating group case, though we didn't set that up with a lattice action as in this case. 

Consider the matrix $$ \label{InvHell} 
M(x,y,z)\eqdef \begin{pmatrix} 1 &x &z \\ 0& 1 &y \\ 0& 0& 1\end{pmatrix} , \text{ with inverse } 
 \begin{pmatrix} 1 &-x &xy\nm z \\ 0& 1 &-y \\ 0& 0& 1\end{pmatrix}. $$
For $R$ a commutative ring, the $3\times 3$  {\sl  Heisenberg group\/} with entries in $R$: 
$$\bH_{R}=\{M(x,y,z)\}_{x,y,z\in
R}.$$ 
\cite[\S4.2.1]{FrH15}  shows the action of $\bZ/3$ extends to the {\sl small Heisenberg group\/} providing this Nielsen class with a non-trivial lift invariant. Even at level 0, that separates the braid orbits into those with 0 lift invariant, and those with lift invariant  in $(\bZ/\ell)^*$.  The lift invariant values grow going up the tower, because the Heisenberg group kernel grows. 

\cite[Prop.~4.18]{FrH15} gives a formula for the lift invariant in this case when $r=4$, the first such formula going beyond the Nielsen classes for $A_n$ and $\bfC$ odd order classes (as in the discussion \sdisplay{\cite{Se90a}}). 

Let  $K_\ell= {{\ell\pm1}\over 6}$, for $\ell>3$ prime.  We did the case $\ell=2$ in Prop.~\ref{A43-2}, but the numerics are different. Here is part of  \cite[Thm.~5.2]{FrH15} on the level 0 components of the Hurwitz space. 
{\sl Double identity Nielsen class representatives\/}  have the form $$(g,g,g_3,g_4), \text{ quite different than the \HM\ reps (Def.~\ref{HMrep}).}$$ 

\begin{thm}[Level 0 Main Result] \label{level0MT} For $ \ell> 3$ prime and $k = 0$  there are $K_\ell$  braid orbits with trivial (0) lift invariant. All are \HM\ braid orbits. Orbits with nontrivial lift invariant are those containing {\sl double identity\/}  cusps. These are each distinguished by the value of that lift invariant. \end{thm} 

The orbits with nontrivial lift invariant are conjugate over $\bQ(e^{2\pi i/\ell})$, but this is not from the \BCL. As with Serre's \OIT, there is another Nielsen class to compare with this one. Indeed, that Nielsen class is $\ni((\bZ/\ell)^4,\bfC_{+3^2-3^2})^{\inn,\rd}$. One new phenomena is the presence of more than one \HM\ braid orbit. The {\sl Weil pairing\/} already appeared in the modular curve case. Here, too.  

\begin{guess} \label{OITgen}  Consider $(G,\bfC)$, with $r=4$, and a prime $\ell$ for which $G$ is $\ell$-perfect. Also, let  $\{\sH_k\}_{k=0}^\infty$ be a \MT\ on $\ni(G,\bfC)^{\inn,\rd}$, corresponding to a braid orbit $O$ on  $\{\ni(G_k,\bfC\}_{k=0}^\infty$,  Then, the geometric monodromy $G_O$ of $O$ is an eventually $\ell$-Frattini group. 

Further, if $\sP_{j'}=\{\bp_k'\in \sH_k\}_{k=0}^\infty$ is a projective sequence of points lying over $j'\in \bar \bQ$ then, $G(\sP_{j'})\cap G_O$ is an eventually $\ell$-Frattini subgroup of $G_O$. 
\end{guess}

As in Def.~\ref{namedclasses}, we use $\CM$ and $\GL_2$ for the respective Nielsen classes $\ni(D_\ell,\bfC_{2^4})$, $\ni((\bZ/\ell)^2,\bfC_{2^4})$ and the decorations that appear with them. 

In both \S\ref{earlyOIT-MT} and in \S\ref{pre95}, under \sdisplay{\cite{Se68}} we call attention to the eventually Frattini Def.~\ref{evenfratt} and its abstraction for the \CM\ and $\GL_2$ cases in \eqref{serreles}. Then, in  \sdisplay{\cite{Fr78}}  we take advantage of the relation between the two distinct Nielsen classes as follows. 
\begin{edesc} \label{CMGL2} \item \label{CMGL2a} The $\GL_2$ covers of the $j$-line are related to the \CM\ covers over the $j$-line, by the former being the Galois closure of the latter. 
\item \label{CMGL2b} For \CM\ or $\GL_2$, the extension of constants indicates that we have the appropriate description of the fiber according to the \OIT.  
\item \label{CMGL2c} In the $\GL_2$ case, in referring to \eql{CMGL2}{CMGL2b}, braids give the geometric elements, $\text{\rm SL}_2(\bZ/\ell^{k\np1})/\lrang{\pm1}$ in this case, of $N_{S_n}(G)/G$. That gives all the geometric monodromy of the $j$-line covers. \end{edesc} 

Comment on \eql{CMGL2}{CMGL2a}: This was what our discussion of \sdisplay{\cite{Fr78}} was about. 
From \eql{CMGL2}{CMGL2c} we \lq\lq see\rq\rq\  the $\GL_2$ geometric monodromy. The most well-known proof of \cite{Se68} is -- in our language -- that $
\{\text{\rm SL}_2(\bZ/\ell^{k\np1}/\lrang{\pm1}\}_{k=0}^{\infty}$ is  $\ell$-Frattini (resp.~eventually $\ell$-Frattini) for $\ell > 3$ (resp.~for all $\ell$). 

Such moduli spaces, affording refined ability to interpret cusps, enable objects of the style of the Tate curves (as in \sdisplay{\cite{Fr78}}{) around, say, the {\sl Harbater-Mumford\/} type cusps. This is compatible with those cusps in the discussions of \sdisplay{\cite[Thm.~3.21]{Fr95} ,  \cite{W98}  and  \cite{DEm05}}.  Nevertheless, these have yet to be sufficiently developed. 

As in the discussion of \sdisplay{\cite{Fr78}} we can expect inexplict versions of Faltings \cite{Fa83} if we can find any version of them at this time. Also, we will find ourselves asking how far into one of the eventually $\ell$-Frattini strands we must go to assure the fiber over  $j'\in \bar \bQ$ has revealed itself. 

\begin{rem}[Rem.~\eqref{braidorbits1} Cont.] \label{braidorbits2} \cite{FrH15} shows regular behavior on any one \MT. Still, there are basically different types of such towers, depending on whether the lift invariant in the tower is invertible $\!\!\mod \ell$  or not. In any case, as we change $\ell$, even at level 0, there are an increasing number of components, many of which are conjugate over $\bQ$. \end{rem}

\providecommand{\bysame}{\leavevmode\hbox to3em{\hrulefill}\thinspace}
\providecommand{\MR}{\relax\ifhmode\unskip\space\fi MR}
\providecommand{\MRhref}[2]{%
\href{http://www.ams.org/mathscinet-getitem?mr=#1}{#2}}
\providecommand{\href}[2]{#2}


\end{document}